\documentclass[11pt]{amsart}

\usepackage{graphicx}
\usepackage{subcaption}
\usepackage{float}
\usepackage[section]{placeins}
\usepackage{pdflscape}
\usepackage{tikz}
\usetikzlibrary{matrix}
\usepackage{multicol}
\usepackage{amssymb,amsthm,amsaddr}
\usepackage{bbm}
\usepackage{graphbox}
\usepackage{upgreek}
\usepackage{tcolorbox}
\usepackage{tabto}
\usepackage{braket}
\usepackage[export]{adjustbox}
\usepackage{nicematrix}
\usepackage{appendix}
\usepackage{circledsteps}
\usepackage{mathalpha} 
\usepackage{enumitem}
\DeclareMathAlphabet{\mathpzc}{OT1}{pzc}{m}{it}

\newcommand{\K}{\mathbb{K}}
\usepackage{mathtools} 
\usepackage{mathrsfs} 
\usepackage{stmaryrd} 
\usepackage{dsfont} 

\usepackage[margin=1in]{geometry}
\usepackage[pagebackref, colorlinks = true, linkcolor = blue, urlcolor  = blue, citecolor = red]{hyperref}
\usepackage{xcolor}
\definecolor{darkgreen}{rgb}{0.09, 0.45, 0.27}
\definecolor{darkred}{rgb}{0.55, 0.0, 0.0}
\usepackage{mathabx,epsfig}
\def\acts{\mathrel{\reflectbox{$\righttoleftarrow$}}}

\newcommand{\ip}[1]{\left\langle #1 \right\rangle}
\newcommand{\RomanNum}[1]{\textup{\uppercase\expandafter{\romannumeral#1\relax}}} 

\renewcommand{\epsilon}{\varepsilon}
\renewcommand{\phi}{\varphi}
\newcommand{\C}{\mathbb{C}}

\newcommand{\bSR}{\mathbb{S}_\R}
\newcommand{\bSC}{\mathbb{S}_\C}

\newcommand{\lnorm}{\lvert \!\vert}
\newcommand{\rnorm}{\vert \! \rvert}

\newcommand{\E}{\mathbb{E}}

\renewcommand{\P}{\mathbb{P}}

\newtheorem{theorem}{Theorem}[section]
\newtheorem{definition}[theorem]{Definition}
\newtheorem{proposition}[theorem]{Proposition}
\newtheorem{corollary}[theorem]{Corollary}
\newtheorem{lemma}[theorem]{Lemma}
\newtheorem{remark}[theorem]{Remark}
\newtheorem{example}[theorem]{Example}

\newcommand{\Sym}{\mathrm{Sym}}
\newcommand{\sym}{\mathrm{sym}}

\newcommand{\R}{\mathbb{R}}
\newcommand{\N}{\mathbb{N}}
\newcommand{\abs}[1]{\left\lvert #1 \right\rvert}
\newcommand{\diff}{\mathop{}\!\mathrm{d}}

\newcommand{\defeq}{\vcentcolon=}

\newcommand{\mc}[1]{\mathcal{#1}}

\DeclareMathOperator{\OO}{O}
\DeclareMathOperator{\oo}{o}

\DeclareMathOperator{\argmin}{argmin}
\DeclareMathOperator{\diag}{diag}

\newcommand{\injnorm}[1]{\lnorm #1 \rnorm_{\mathrm{inj}}}
\newcommand{\injnormR}[1]{\lnorm #1 \rnorm_{\mathrm{inj},\R}}
\newcommand{\injnormC}[1]{\lnorm #1 \rnorm_{\mathrm{inj},\C}}
\newcommand{\injnormK}[1]{\lnorm #1 \rnorm_{\mathrm{inj},\K}}

\makeatletter
 \newcommand{\linkdest}[1]{\Hy@raisedlink{\hypertarget{#1}{}}}
\makeatother

\title{A moment-based approach to the injective norm of random tensors}

\author{Stephane Dartois}
\address{\'Ecole Polytechnique, Institut Polytechnique de Paris, Centre de Mathématiques Laurent Schwartz, Palaiseau, France}
\email{stephane.dartois@polytechnique.edu}
\author{Benjamin M\textsuperscript{c}Kenna}
\address{Georgia Institute of Technology, School of Mathematics, Atlanta, GA, United States}
\email{mckenna@math.gatech.edu}
\begin{document}

\begin{abstract}

In this paper, we present a technically simple method to establish upper bounds on the expected injective norm of real and complex random tensors. Our approach is somewhat analogous to the moment method in random matrix theory, and is based on a deterministic upper bound on the injective norm of a tensor which might be of independent interest. Compared to previous approaches to these problems (spin-glass methods, epsilon-net techniques, Sudakov-Fernique arguments, and PAC-Bayesian proofs), our method has the benefit of being nonasymptotic, relatively elementary, and applicable to non-Gaussian models. We illustrate our approach on various models of random tensors, recovering some previously known (and conjecturally tight) bounds with simpler arguments, and presenting new bounds, some of which are provably tight. From the perspective of statistical physics, our results yield rigorous estimates on the ground-state energy of real and complex, possibly non-Gaussian, spin glass models. From the perspective of quantum information, they establish bounds on the geometric entanglement of random bosonic states and of random states with bounded multipartite Schmidt rank, both in the thermodynamic limits as well as the regimes of large local dimensions.
\end{abstract}

\maketitle

\vspace*{0.05in}

\noindent \emph{Date:} March 1, 2026

\vspace*{0.05in}

\noindent \hangindent=0.2in \emph{Keywords and phrases:} Injective norm, Spectral norm, Random tensors, Large moments, Gaussian tensors, Geometric entanglement, Spherical spin glass

\vspace*{0.05in}


\vspace*{5mm}

\tableofcontents


\section{Introduction}
The goal of this work is to develop a new method to upper bound the injective norm (a tensor analogue of the operator norm of a matrix) of random tensors. 
Our approach is a \emph{moment method}, inspired by the classical moment method in random matrix theory, which is frequently used (dating back to \cite{furedi1981eigenvalues}) to study the operator norm of a random matrix. (For treatments of this method in the matrix context, see, \textit{e.g.}, \cite{Tao2012, van2017structured}.) However, our method is technically different: it relies on what we think to be new deterministic inequalities controlling even powers of the injective norm in terms of moments of projections onto random rank-one tensors. In the settings considered here, this strategy is conceptually straightforward and often technically lighter than existing approaches for random tensor injective norms \cite{auffinger2013random,TomSuz2014,dartois2024injective,sasakura2024signed,bandeira2024geometric,boedihardjo2024injective,bates2025balanced,dartois2025injective}. Moreover, we believe this mechanism is robust to input distributions and to several structural constraints on the tensor, and it should transfer to other settings where injective norms (or closely related multilinear forms) arise. As proof of concept of this method, we give upper bounds for the injective norms of various models of random tensors. In some regimes, we can recover known upper bounds that are believed to be tight, with significantly simpler arguments; in other regimes, including for models that appear to be out of reach of previous techniques, we find new bounds. We now introduce formally the problem and models studied in the paper before discussing our main results and their context.


\subsection{Problem set-up and models}
Let $\K\in\{\R,\C\}$ be the field of either real or complex numbers. Let $p\in \mathbb{N}^*$ be a positive integer, $d\in \mathbb{N}^*$ and $d_1,\ldots, d_p\in \mathbb{N}^*$ be a family of $p$ positive integers. For all $x,y\in \K^{d_i}$ denote their Hermitian product by $\langle x,y\rangle=\sum_{k=1}^{d_i}\bar{x}_k y_k$. For $T\in \bigotimes_{i=1}^p\K^{d_i}$ a tensor, we define its injective norm $\injnormK{T}$ as the quantity 
\begin{equation}
    \injnormK{T}:=\max_{\substack{x_i\in \K^{d_i}\\\lnorm x_i\rnorm_2=1}}\lvert\langle T,x_1\otimes \ldots\otimes x_p\rangle \rvert = \max_{\substack{x_i\in \K^{d_i}\\\lnorm x_i\rnorm_2=1}} \abs{ \sum_{i_1=1}^{d_1} \cdots \sum_{i_p=1}^{d_p} \overline{T_{i_1,\ldots,i_p}} (x_1)_{i_1} \cdot \ldots \cdot (x_p)_{i_p}}.
\end{equation}
This is the main quantity of interest in this paper; see Section \ref{subsec:context} below for more discussion and motivation. In fact, we study the average injective norm of random tensors $T$ from different ensembles which we define below. We start with the following definition.
\begin{definition}
\label{def:injective-norm}
 We say that a $\R$-valued random variable $X$ is  $\R$-rigidly sub-Gaussian if it is symmetric, \textit{i.e.} $X\overset{\text{law}}{=}-X$, and its non-vanishing moments are dominated by those of a centered Gaussian random variable with the same variance, that is, $\E[X^{2k}]\le (2k-1)!!\E[X^2]^{k}.$ \\

 \noindent We say that a $\C$-valued random variable $Y$ is $\C$-rigidly sub-Gaussian if it is circularly symmetric (and in particular centered), \textit{i.e.} for all $\theta \in \R,$ $Y\overset{\text{law}}{=}e^{i\theta}Y$, and the even moments of its modulus are dominated by those of a circularly symmetric complex Gaussian random variable with the same variance, that is, $\E[\lvert Y\rvert^{2k}]\le k! \E[\lvert Y\rvert^{2}]^k.$
\end{definition}

In recent papers such as \cite{GuiHus2020,CooDucGui2023,BobChiGot2024}, the terms ``sharp sub-Gaussian'' and ``strict sub-Gaussian'' have been used for random variables whose moment-generating functions are upper-bounded by those of a Gaussian with the same variance. We remark that rigidly sub-Gaussian variables are a subclass of these.\\

\noindent{\bf \linkdest{model:ak}Model $A_\K$.}
For $\K \in \{\R,\C\}$, we say that an \emph{asymmetric} random tensor $T$ comes from ``Model $A_\K$'' ($A$ for ``asymmetric'') if it lies in $\K^{d_1} \otimes \cdots \otimes \K^{d_p}$, and has i.i.d. entries which are $\K$-rigidly sub-Gaussian, which are normalized as $\E[|T_{1,\ldots,1}|^2] = \frac{1}{(d_1 \cdots d_p)^{1/p}}$.
\begin{example}
We give examples of random tensors in \hyperlink{model:ak}{Models $A_\R$} and \hyperlink{model:ak}{$A_\C$}:
\begin{itemize}
    \item Consider a tensor $A\in \bigotimes_{i=1}^p\R^{d_i}$ whose entries are i.i.d. Rademacher variables, or i.i.d. $\text{Uniform}([-\sqrt{3},\sqrt{3}])$ variables, or of course i.i.d. standard Gaussian variables. The rescaled tensor $T=\frac1{(d_1\ldots d_p)^{\frac1{2p}}}A$ is a tensor of \hyperlink{model:ak}{Model $A_\R$}. 
    \item For a complex example, consider a tensor $A\in \bigotimes_{i=1}^p \C^{d_i}$, whose entries are i.i.d. uniform on the complex unit circle (also called Steinhaus random variables). Then $T=\frac1{(d_1\ldots d_p)^{\frac1{2p}}}A$ is a tensor of \hyperlink{model:ak}{Model $A_\C$}. Moreover, the Frobenius norm of $T$ is deterministic, $\lnorm T\rnorm_{\textup{HS}}^2=\sum_{i_1,\ldots,i_p}\lvert T_{i_1,\ldots,i_p} \rvert^2=(d_1\ldots d_p)^{1-1/p}.$ Hence, $(d_1\ldots d_p)^{\frac12\left(\frac1p-1\right)}T$ is a random quantum state, which we call a Steinhaus random state. 
\end{itemize}
\end{example}
\noindent{\bf \linkdest{model:sc}Model $S_\C$.}
Let $\gamma$ be a $\C$-rigidly sub-Gaussian random variable with variance $\E[|\gamma|^2] = 1/d$. We say that a \emph{symmetric} random tensor $B\in \Sym_p(\C^d)$ comes from ``Model $S_\C$'' ($S$ for ``symmetric''), if:
\begin{itemize}
 \item its entries are independent up-to-symmetry,
 \item $B_{i_1,\ldots,i_p}\overset{\text{law}}{=}\sqrt{\frac{\lvert\text{Stab}_{\mathfrak{S}_p}((i_1,\ldots,i_p))\rvert}{p!}}\gamma,$ where $\text{Stab}_{\mathfrak{S}_p}((i_1,\ldots,i_p))$ denotes the subgroup of the permutation group over $p$ elements $\mathfrak{S}_p$ stabilizing the $p-$tuple $(i_1,\ldots,i_p)$. Concretely, if $(i_1,\ldots,i_p)$ has $k$ distinct elements with respective multiplicities $m_1,\ldots,m_k$, then $B_{i_1,\ldots,i_p} \overset{\text{law}}{=} \sqrt{\frac{m_1! \cdots m_k!}{p!}} \gamma$.
\end{itemize}
\noindent{\bf \linkdest{model:stildec}Model $\widetilde{S}_\C$.}
Consider $T\in (\C^d)^{\otimes p}$ a random tensor of \hyperlink{model:ak}{Model $A_\C$}. We say a symmetric random tensor $\widetilde{B}\in \text{Sym}_p(\C^d)$ belongs to model $\widetilde{S}_\C$, if $\widetilde{B}:=P_{\sym ,p}^{(d)}(T)$, where $P_{\sym,p}^{(d)}:(\C^d)^{\otimes p}\rightarrow \text{Sym}_p(\C^d)\subseteq(\C^d)^{\otimes p}$ is the projector on the symmetric subspace. Equivalently, $\widetilde{B}$ is the average over all permutations of $T$, that is $\widetilde{B} = \frac1{p!}\sum_{\sigma \in \mathfrak{S}_p} \sigma(T)$, where the permutation group $\mathfrak{S}_p$ acts as described in the Notations section below.  
\begin{remark}
    We mention that, while Models \hyperlink{model:sc}{$S_\C$} and \hyperlink{model:stildec}{$\widetilde{S}_\C$} are the same in the Gaussian case, they are different in general; see Remark \ref{rem:SCvstildeSC} for details.
\end{remark}

\noindent{\bf \linkdest{model:bc}Model $B_\C$.}
Let $R\in \mathbb{N}$ and let $x_1^{(i)},\ldots,x_p^{(i)}$ be, for each $1\le i\le R$, a family of $p$ independent vectors $x_m^{(i)}$ with i.i.d. $\C$-rigidly sub-Gaussian entries with variance set to $1/d$, independent for different $i$'s. We consider the random tensor $T=\sum_{i=1}^Rx_1^{(i)}\otimes \ldots\otimes x_p^{(i)}.$ The entries of $T$ are identically distributed but correlated, and their variance is $\E(\lvert T_{i_1,\ldots,i_p}\rvert^2)=\frac{R}{d^p}$. We call this model the complex bounded-rank random tensor model $B_\C$ ($B$ for ``bounded''), or simply the bounded-rank model.


\subsection{Results}
Our main results can be informally recast as follows:
\begin{theorem}
\textbf{(Informal main theorems)} If $T$ is any deterministic tensor, we give an upper bound for its injective norm in terms of its averaged projections on uniformly random vectors on the sphere (Theorem \ref{thm:deterministic_bound}). As applications of this result, we give upper bounds on random tensors from the Models \hyperlink{model:ak}{$A_\K$} (Theorem \ref{thm:asymmetric}), \hyperlink{model:sc}{$S_\C$} and \hyperlink{model:stildec}{$\widetilde{S}_\C$} (Theorem \ref{thm:symm_complex_tensors_bound}), and \hyperlink{model:bc}{$B_\C$} (Theorem \ref{thm:bounded_rank}).
\end{theorem}
These bounds in many cases improve on the state of the art; see Appendix \ref{app:prior-bounds-comparison} for a detailed comparison to existing methods. We also mention that the real-symmetric case of our main result is partially comparable to well-known bounds of Lata\l a \cite{Lat2006} on moments of Gaussian chaoses; see Remark \ref{rem:latala}. The proof of those results only relies on elementary analysis, mostly involving standard calculus methods to study asymptotics when the order of the moment grows to infinity. Standard combinatorial devices are also used in this moment method, which are only needed to extend our results beyond the Gaussian case. 


\subsection{Context}
\label{subsec:context}

The injective norm of tensors appears in different contexts, and these motivate our work on the topic. We proceed to describe some of them; see also the introduction to our prior work \cite{dartois2024injective} for a more in-depth discussion, including appearances of the injective norm in statistical physics (specifically in spin glasses) and data analysis (specifically tensor PCA).\\

\noindent{\bf Quantum information.}
The injective norm is an important quantity in quantum information, where it is known under the name of geometric entanglement; see \textit{e.g.} \cite{wei2003geometric,aulbach2010maximally,aulbach2010geometric,aubrun2017alice,dartois2025injectiveCSS}. 
In this context, the tensor is a pure quantum state $\ket{\psi}$, hence has Frobenius norm one, and the injective norm is the maximal overlap of $\ket{\psi}$ with separable states, therefore providing a measure of multipartite entanglement. Our work fits in the broader program of studying generic entanglement properties of quantum states, where one aims at probabilistic asymptotic statements and entanglement thresholds. While this program has produced an extensive literature in the case of bipartite entanglement (i.e., the matrix case $p = 2$ above), studying bipartite entanglement measures and criteria for a variety of different ensembles of quantum states 
\cite{sommers2004statistical, aubrun2012phase, aubrun2012realigning, aubrun2014entanglement, hayden2016holographic, nakata2020generic, BianchiKieburg2021, dong2021holographic, dartois2022entanglement, bianchi2022volume, pastur2024entanglement, cheng2024random}, the multipartite case (i.e., the tensor case $p \geq 3$ above) is much less understood (see however \cite{GroFlaEis2009,friedland2018most,penington2023fun, akers2024reflected}). It then seems natural to study the generic properties of the injective norm, and its behavior for a variety of random states. This seems even more so given that computing the injective norm of a generic state is a NP-hard problem \cite{hillar2013most}. In this context, our main theorems give lower bounds on the entanglement of random quantum states.\\

\noindent{\bf Hypergraphs spectral properties and applications.} 
In the classical area of spectral graph theory, one learns about graphs by studying their associated matrices (such as their adjacency and Laplacian matrices). Random graphs correspond to random matrices. However, uniform (random) \emph{hypergraphs} are associated instead with (random) \emph{tensors}, and various lines of research have developed hypergraph-theoretic interpretations of the injective norm. For example, the second eigenvalue of a matrix is known to determine its properties as a so-called ``expander graph''; \cite{friedman1995second} used the injective norm to develop a notion of the second eigenvalue of a tensor, and related this to the notion of expander hypergraphs. Their motivation stems from complexity theory, as hypergraphs that are sufficiently good expanders can be turned into strong dispersers (\textit{i.e.} a particular type of maps between two languages). It is known through a theorem of Sipser \cite{Sipser1988} that the existence of such dispersers implies either the collapse of $\mathbf{RP}$ (for Randomized Polynomial time) class to $\mathbf{P}$, or a strong trade-off between time and space. More recently and parallel to this question, the injective norm appeared in the problem of community detection (see \cite{chan2018spectral,lin2025phase} and references therein). There it plays the role of a ``spectral'' probe used to detect communities in networks.\\

\noindent{\bf Generalization of random matrix theory.}
Over the past decades, random matrix theory has become a central area of mathematics, with connections to mathematical physics, probability, statistics, numerical linear algebra, and theoretical computer science \cite{mehta2004random,forrester2010log,anderson2010introduction,potters2020first}. Among its core themes, one salient question -- in both theory and applications -- is the behavior of the top eigenvalue or singular value, including its universality \cite{ErdYau2017} and limiting laws \cite{bai1988necessary,tracy1994level,soshnikov1999universality,benarous2001aging}. From this perspective, the case $p=2$ of the injective norm is the top singular value of a matrix, and higher values of $p$ provide a natural extension to tensors. As was noted in \cite{dartois2025injective}, it is therefore natural to ask what can be said about the injective norm, with what precision and what degree of universality, mimicking the corresponding program in the theory of random matrices (see \textit{e.g.} \cite{subag2017extremal,huang2023constructive,stojnic2025ground} for some related results, due to statistical physics and the landscape complexity program).
Among the many difficulties one faces, the lack of spectral decomposition and the limitations of tools dedicated to the excursion probability of non-Gaussian processes are the most obvious ones. Our large moment method aims to mitigate those difficulties, proposing a new approach, with some level of robustness, which we hope can support a somewhat systematic study of the injective norm of random tensors.\\


\subsection*{Organization of the paper}

In Section \ref{sec:deterministic_bound}, we state and prove our main upper bounds on the injective norm of deterministic tensors. In Sections \ref{sec:application_asymmetric}, \ref{sec:application_symmetric}, and \ref{sec:bounded_rank}, we apply these bounds to estimate the injective norm of random tensors from Models \hyperlink{model:ak}{$A_\K$}; \hyperlink{model:sc}{$S_\C$} and \hyperlink{model:stildec}{$\widetilde{S}_\C$}; and \hyperlink{model:bc}{$B_\C$}, respectively. The paper ends with four appendices. Appendix \ref{app:numerics_experiment} gives numerics supporting our bounds. Appendix \ref{app:rank_property} gives an algebraic-geometric argument showing that, in the Lebesgue absolutely continuous case of the bounded-rank model \hyperlink{model:bc}{$B_\C$}, the rank is not just upper-bounded by $R$ (which is immediate from the construction) but is actually equal to $R$ almost surely. Sections \ref{sec:application_asymmetric}, \ref{sec:application_symmetric}, and \ref{sec:bounded_rank} give asymptotic bounds as the tensor grows (either $d \to \infty$ or $p \to \infty$), but Appendix \ref{app:non-asymptotic-upper-bound} illustrates how our method can also give non-asymptotic bounds for $d$ and $p$ fixed, which can additionally be computed in a numerically efficient way. Finally, in Appendix \ref{app:prior-bounds-comparison}, we compare our new method to previous ones (Kac--Rice arguments, Sudakov--Fernique arguments, epsilon-net arguments, and PAC-Bayesian arguments) and show that our bound in many cases improves on the previous state of the art.


\subsection*{Notations}
Throughout the course of this paper, we denote by $$\bSR^{d-1}:=\{x\in \R^d:\lnorm x\rnorm_2=1\},$$
the $(d-1)$-dimensional unit sphere in $\R^d$, that is the sphere of vectors $x\in \R^d$ whose $2$-norm is set to $1.$ We denote by 
$$\bSC^{d-1}:=\{x\in \C^d:\lnorm x\rnorm_2=1\},$$
the complex unit sphere (which is a real manifold with $2d-1$ real dimensions).

For all $n\in \mathbb{N},$ we denote by $\mathfrak{S}_n$ the permutation group over $n$ elements. We often consider its action on tensors: for all $x_1,\ldots,x_n\in \K^d,\ \sigma \in \mathfrak{S}_n,$ the action on simple tensors is $\sigma(x_1\otimes \ldots\otimes x_n)=x_{\sigma(1)}\otimes \ldots \otimes x_{\sigma(n)}$, and it extends linearly. For a tensor $T \in \K^{d_1} \otimes \cdots \otimes \K^{d_p}$, we write $\|T\|_{\textup{HS}}$ for the Frobenius (or Hilbert-Schmidt) norm $\|T\|_{\textup{HS}}^2 = \sum_{i_1=1}^{d_1} \cdots \sum_{i_p=1}^{d_p} \abs{T_{i_1,\ldots,i_p}}^2$.

{
  \let\oldaddcontentsline\addcontentsline
  \renewcommand{\addcontentsline}[3]{}
\section*{Acknowledgments}
We wish to thank Yizhe Zhu for helpful discussions and references, as well as Steven Heilman for interesting suggestions and related references. We are also grateful to the  School of Mathematics at the Georgia Institute of Technology and the Institut de Mathématiques de Bordeaux for their hospitality during parts of this work. The work of S.D. was partly supported by the ANR grants ANR-25-CE40-1380 and ANR-25-CE40-5672. The work of B.M. was partly supported by an AMS-Simons Travel Grant. 
}


\section{Deterministic upper bound}\label{sec:deterministic_bound}

The purpose of this section is to prove the following theorem. 

\begin{theorem}\label{thm:deterministic_bound}
For all $k,p \in \mathbb{N}^*$, the following upper bounds hold:
\begin{enumerate}
    \item[a.] Let $T\in \bigotimes_{i=1}^p\R^{d_i}$ be a real tensor. 
    Let $u_i\in \bSR^{d_i-1}$ be independent uniform random vectors on the real unit spheres of dimensions $d_i$. One has
    \begin{equation}\label{eq:real_inj_norm_bounded}
       \injnormR{T}^{2k}\le \left(\prod_{i=1}^{p}\frac{2^k\Gamma(\frac{d_i}{2}+k)}{(2k-1)!!\Gamma(\frac{d_i}{2})}\right) \E_{u}\left( \langle T,u_1\otimes \ldots\otimes u_p\rangle^{2k}\right). 
    \end{equation}
    \item[b.] Let $T\in\bigotimes_{i=1}^p\C^{d_i}$ be a complex tensor. 
    Let $u_i\in \mathbb{S}_\C^{d_i-1}$ be uniform random vectors on the complex unit spheres of dimensions $d_i$. One has
    \begin{equation}\label{eq:inj_norm_bounded}
       \injnormC{T}^{2k}\le \left( \prod_{i=1}^p\binom{d_i+k-1}{k} \right) \E_{u}\left( \lvert \langle T,u_1\otimes \ldots\otimes u_p\rangle\rvert^{2k}\right).
    \end{equation}
    \item[c.] Let $B\in \text{Sym}_p(\C^d)\subseteq (\C^d)^{\otimes p}$ be a complex symmetric tensor. Let $u\in \mathbb{S}_\C^{d-1}$ be a uniform random vector on the complex unit sphere of dimension $d$. One has
 \begin{equation}\label{eq:symmetric_inj_norm_bounded}
     \injnormC{B}^{2k}\le \binom{d+pk-1}{pk}\E_u(\lvert \langle B,u^{\otimes p}\rangle\rvert^{2k}).
 \end{equation}
\end{enumerate}
\end{theorem}

We start with the proof of the complex claim, which is a bit simpler.


\subsection{Proof of Theorem \ref{thm:deterministic_bound}: Complex case}
The main claim is that
    \begin{equation}
    \label{eqn:complex-asymmetric-main-claim}
        \E_{u}\left( \lvert \langle T,u_1\otimes \ldots\otimes u_p\rangle\rvert^{2k}\right)=\left(\prod_{i=1}^p\binom{d_i+k-1}{k}\right)^{-1}\left\langle T^{\otimes k},\left(\bigotimes_{i=1}^p P_{\text{sym},k}^{(d_i)}\right) T^{\otimes k}\right\rangle,
    \end{equation}
    where $P_{\text{sym}, k}^{(d_i)}=\frac1{k!}\sum_{\sigma\in\mathfrak{S}_k} \sigma$ is the projector $P_{\text{sym}, k}^{(d_i)}:(\C^{d_i})^{\otimes k}\rightarrow \text{Sym}_k(\C^{d_i})$ on the symmetric subspace. (Recall, here $\sigma\acts (\C^{d_i})^{\otimes k}$ acts by exchanging tensor factors, {\it i.e.} $\sigma(x_1\otimes\ldots\otimes x_k)= x_{\sigma(1)}\otimes \ldots \otimes x_{\sigma(k)}$ and extends linearly.) Indeed, we first rewrite
    \begin{align*}
        \abs{\ip{T,u_1 \otimes \cdots \otimes u_p}}^{2k} &= \ip{T^{\otimes k}, \ket{ (u_1 \otimes \cdots \otimes u_p)^{\otimes k} } \bra{ (u_1 \otimes \cdots \otimes u_p)^{\otimes k}} T^{\otimes k}} \\
        &= \ip{T^{\otimes k}, \left( \bigotimes_{i=1}^p \ket{u_i^{\otimes k}} \bra{u_i^{\otimes k}} \right) T^{\otimes k}},
    \end{align*}
    then average over $u$ and use independence to interchange expectation and $\bigotimes_{i=1}^p$. Then we notice that
    \begin{equation}\label{eq:uniform-to-Gaussian}
    \E_{u_i} \lvert u_i^{\otimes k}\rangle\langle u_i^{\otimes k}\rvert=\frac1{\E_{\phi_i}(\lnorm \phi_i\rnorm_2^{2k})}\E\lvert \phi_i^{\otimes k}\rangle\langle \phi_i^{\otimes k}\rvert,
    \end{equation}
    where $\phi_i\in \C^{d_i}$ is a standard complex random normal vector, i.e., its entries are i.i.d., and their real and imaginary parts are i.i.d. centered Gaussians of variance $1/2$. This holds because $u_i$ has the same distribution as $\varphi_i/\|\varphi_i\|_2$, and the norm and direction of a normal vector with independent identically distributed entries are independent random variables. We have $\E\lvert \phi_i^{\otimes k}\rangle\langle \phi_i^{\otimes k}\rvert=\sum_{\sigma\in\mathfrak{S}_k}\sigma = k! P^{(d_i)}_{\text{sym},k}$ as a consequence of say Wick-Isserlis theorem.\footnote{This can be shown in many different ways.} Since $\E_{\varphi_i}(\|\varphi_i\|_2^{2k}) = \Gamma(d+k)/\Gamma(k)$, this completes the proof of \eqref{eqn:complex-asymmetric-main-claim}.
    
    Now, for all unit-norm $v$ in the image $\text{Im}\left(\bigotimes_{i=1}^pP_{\text{sym},k}^{(d_i)}\right)$, we have 
    \begin{align*}
        \abs{\ip{T^{\otimes k}, v}}^2 &= \abs{\ip{T^{\otimes k}, \left(\bigotimes_{i=1}^pP_{\text{sym},k}^{(d_i)}\right) v}}^2 = \abs{\ip{\left(\bigotimes_{i=1}^pP_{\text{sym},k}^{(d_i)}\right) T^{\otimes k}, v}}^2 \\
        &\leq \left\|\left(\bigotimes_{i=1}^pP_{\text{sym},k}^{(d_i)}\right) T^{\otimes k}\right\|^2 \|v\|^2 = \ip{ \left(\bigotimes_{i=1}^pP_{\text{sym},k}^{(d_i)}\right) T^{\otimes k}, \left(\bigotimes_{i=1}^pP_{\text{sym},k}^{(d_i)}\right) T^{\otimes k}} \\
        &= \ip{ T^{\otimes k}, \left(\bigotimes_{i=1}^pP_{\text{sym},k}^{(d_i)}\right) T^{\otimes k}},
    \end{align*}
    since $\left(\bigotimes_{i=1}^pP_{\text{sym},k}^{(d_i)}\right)$ is a projector. In particular, for any choice of $x_i \in \C^{d_i}$ of unit norm, setting $v=x_1^{\otimes k}\otimes x_2^{\otimes k}\otimes \ldots \otimes x_p^{\otimes k}$ leads to 
    \[
        \left\langle T^{\otimes k},\left(\bigotimes_{i=1}^pP_{\text{sym},k}^{(d_i)}\right) T^{\otimes k}\right\rangle\ge\left\lvert\langle T^{\otimes k},x_1^{\otimes k}\otimes x_2^{\otimes k}\otimes \ldots \otimes x_p^{\otimes k}\rangle \right\rvert^2=\left\lvert\langle T,x_1\otimes x_2\otimes \ldots \otimes x_p\rangle \right\rvert^{2k}.
    \]
    Then, maximizing over the choice of $x_i$, one finds 
    \[
        \left\langle T^{\otimes k},\left(\bigotimes_{i=1}^pP_{\text{sym},k}^{(d_i)}\right) T^{\otimes k}\right\rangle\ge\injnormC{T}^{2k}.
    \]
    This proves the claim. \qed


\subsection{Proof of Theorem \ref{thm:deterministic_bound}: Symmetric complex case}
We first recall that for all $B\in \text{Sym}_p(\C^d)$, the injective norm is realized as $\injnormC{B}=\max_{\substack{x\in \C^d\\ \lnorm x\rnorm_2=1}} \lvert\langle B,x^{\otimes p}\rangle \rvert$ (this is a classical result, sometimes called Banach's theorem after \cite{Ban1938}, although \cite{Wat1990} has antedated the result to \cite{Kel1928,vanSch1935}).
Taking note of this fact, the proof of the symmetric version uses similar arguments as the complex case proof of Theorem \ref{thm:deterministic_bound}. One change is that, since we only have one $u$ instead of $p$ of them, we replace $\otimes_{i=1}^p P^{(d_i)}_{\text{sym},k}$ with $P^{(d)}_{\text{sym},pk}$; thus \eqref{eqn:complex-asymmetric-main-claim} is replaced with $\E_u(\lvert \langle B,u^{\otimes p}\rangle\rvert^{2k})=\frac{(pk)!}{\E(\rnorm\phi\lnorm_2^{2pk})}\langle B^{\otimes k}, P_{\text{sym}, pk}^{(d)} B^{\otimes k}\rangle$. To obtain a lower bound we use the same linear algebra trick. For any unit-norm $v\in \text{Im}\left(P_{\text{sym}, pk}^{(d)}\right)$, we have $\langle B^{\otimes k}, P_{\text{sym}, pk}^{(d)} B^{\otimes k}\rangle\ge \langle B^{\otimes k}, vv^\dagger B^{\otimes k}\rangle$. In particular, for any unit-norm $x\in \C^d$ we choose $v=x^{\otimes pk}\in \text{Im}\left(P_{\text{sym}, pk}^{(d)}\right)$ to obtain $\langle B^{\otimes k}, P_{\text{sym}, pk}^{(d)} B^{\otimes k}\rangle\ge \left\lvert\langle B^{\otimes k}, x^{\otimes p k}\rangle\right\rvert^2$. We conclude by maximizing over $x$ to recover the injective norm. \qed


\subsection{Proof of Theorem \ref{thm:deterministic_bound}: Real case}
We start with a preliminary lemma. 
\begin{lemma}\label{lem:gaussian_deterministic_bound}
    Let $p\ge 1$ be an integer, $T\in \bigotimes_{i=1}^{p}\R^{d_i}$ be a fixed real tensor and $x_1\in \bSR^{d_1-1},\ldots, x_p\in \bSR^{d_p-1}$ be any family of vectors in the unit spheres. Additionally, consider a family of $p$ independent standard Gaussian vectors $\phi_1\sim\mathcal{N}_\R(0,\mathbbm{1}_{d_1}),\ldots, \phi_p\sim \mathcal{N}_\R(0,\mathbbm{1}_{d_p}).$ Then, it holds that
    \begin{equation}\label{eq:real_lower_bound_avg}
        \E(\langle T,\phi_1\otimes \ldots \otimes \phi_p\rangle^{2k})\ge((2k-1)!!)^p \langle T,x_1\otimes \ldots \otimes x_p\rangle^{2k}.
    \end{equation}
\end{lemma}
\begin{proof}
We prove the result by recursion on $p$.\\

\noindent{\it Case $p=1.$} In this case $T$ is simply a vector in $\R^d$, and $\langle T, \phi_1\rangle\sim \mathcal{N}_{\R}(0,\lnorm T \rnorm_2^2)$. Therefore
\begin{equation}
\label{eqn:injective-norm-of-vectors}
    \E(\langle T, \phi_1\rangle^{2k})=(2k-1)!!\lnorm T\rnorm_2^{2k}.
\end{equation}
Notice then that $\lnorm T\rnorm_2=\max_{x\in \bSR^{d-1}}\lvert \langle T, x\rangle\rvert,$ hence $\E(\langle T, \phi_1\rangle^{2k})\ge(2k-1)!! \langle T, x\rangle^{2k}.$ This proves equation \eqref{eq:real_lower_bound_avg} for $p=1$.\\

\noindent{\it Induction.}  Assume now that equation \eqref{eq:real_lower_bound_avg} is correct for a given $p\ge 1$. We claim it implies that it is true for $p+1$. To prove this, consider $T\in \bigotimes_{i=1}^{p+1}\R^{d_i}$ and define $T_{(\phi_1)}\in \bigotimes_{i=2}^{p+1}\R^{d_i}$ the tensor contracted on its first slot, that is, $T_{(\phi_1)}=[\phi_1^t\otimes (\bigotimes_{i=2}^{p+1}\mathbbm{1}_{d_i})] T,$ or component-wise $(T_{(\phi_1)})_{i_2,\ldots,i_{p+1}}=\sum_{i_1=1}^{d_1}T_{i_1,i_2,\ldots,i_{p+1}}(\phi_1)_{i_1}$. In order to shorten notations, denote the random variables $\Phi=\phi_2\otimes\ldots \otimes \phi_{p+1}$ and $z=\phi_1\otimes \Phi$, so that in particular $\langle T,\phi_1\otimes \ldots \otimes \phi_{p+1}\rangle =\langle T, z\rangle.$ Then notice that the conditional expectation given $\phi_1$ satisfies
\begin{equation*}
\E(\langle T, z\rangle^{2k}|\phi_1)=\E(\langle T_{(\phi_1)},\Phi\rangle^{2k}|\phi_1)\ge ((2k-1)!!)^p\langle T_{(\phi_1)},x_2\otimes \ldots \otimes x_{p+1}\rangle^{2k},
\end{equation*}
where for the last inequality we use the induction assumption. Then
\begin{equation}\label{eq:conditioning-trick}
\E(\langle T, z\rangle^{2k})=\E_{\phi_1}[\E(\langle T, z\rangle^{2k}|\phi_1)]\ge \E_{\phi_1}\left(((2k-1)!!)^p\langle T_{(\phi_1)},x_2\otimes \ldots \otimes x_{p+1}\rangle^{2k}\right).
\end{equation}
Now define\footnote{In order to keep notations readable we do not display the dependence in the unit norm vectors $x_2,\ldots,x_p$ of $\Theta$, however the reader should keep this implicit dependence in mind.} $\Theta\in \R^{d_1}$ the vector obtained from $T$ by contracting its last $p$ slots with vectors $x_2,\ldots, x_{p+1}$. More formally, $\Theta=[\mathbbm{1}_{d_1}\otimes x_2^t\otimes\ldots\otimes x_{p+1}^t] T$, or again, component-wise, $\Theta_{i_1}=\sum_{i_2=1}^{d_2}\ldots\sum_{i_{p+1}=1}^{d_{p+1}}T_{i_1,i_2,\ldots i_{p+1}}(x_2)_{i_2}\ldots (x_{p+1})_{i_{p+1}}.$ Then,
\begin{equation}\label{eq:reduction-to-base-case}
\E_{\phi_1}\left(\langle T_{(\phi_1)},x_2\otimes \ldots \otimes x_{p+1}\rangle^{2k}\right)=\E_{\phi_1}(\langle \Theta,\phi_1\rangle^{2k})\ge (2k-1)!!\langle T,x_1\otimes \ldots\otimes x_{p+1}\rangle^{2k}
\end{equation}
where the last inequality follows from the $p=1$ case and the remark that $\langle \Theta, x_1\rangle =\langle T,x_1\otimes \ldots\otimes x_{p+1}\rangle$ by definition of $\Theta$. Bringing together equations \eqref{eq:conditioning-trick} and \eqref{eq:reduction-to-base-case} leads to the claim.
\end{proof}
We use the above lemma for the proof of inequality \eqref{eq:real_inj_norm_bounded}:

\begin{proof}[Proof of real inequality \eqref{eq:real_inj_norm_bounded}.]
As in the complex case, we use that $u_j\overset{\text{law}}{=}\frac{\phi_j}{\lnorm \phi_j\rnorm_2},$ where the $\phi_j$ are independent standard real Gaussian vectors $\mathcal{N}_\R(0,\mathbbm{1}_{d_j})$, to obtain
\begin{equation}
\label{eqn:uniform-to-gaussian-real-case}
    \E_u(\langle T, u_1\otimes \ldots \otimes u_p\rangle^{2k})=\frac1{\prod_{i=1}^p\E_{\phi_i}(\lnorm \phi_i \rnorm_2^{2k})}\E_{\phi}(\langle T,\phi_1\otimes\ldots \phi_p\rangle^{2k}).
\end{equation}
For any choice of unit norm $x_{i}\in \bSR^{d_i-1}, i\in [p]$, applying Lemma \ref{lem:gaussian_deterministic_bound} leads to
$$\E_u(\langle T, u_1\otimes \ldots \otimes u_p\rangle^{2k})\ge\frac{((2k-1)!!)^p}{\prod_{i=1}^p\E_{\phi_i}(\lnorm \phi_i \rnorm_2^{2k})}\langle T,x_1\otimes\ldots x_p\rangle^{2k},$$
which after maximizing the right-hand side over the choices of $x_i$ results in
$$\E_u(\langle T, u_1\otimes \ldots \otimes u_p\rangle^{2k})\ge\frac{((2k-1)!!)^p}{\prod_{i=1}^p\E_{\phi_i}(\lnorm \phi_i \rnorm_2^{2k})}\injnormR{T}^{2k}.$$
Therefore, we obtain
$$\frac{((2k-1)!!)^p}{\prod_{i=1}^p\E_{\phi_i}(\lnorm \phi_i \rnorm_2^{2k})}\injnormR{T}^{2k}\le \E_u(\langle T, u_1\otimes \ldots \otimes u_p\rangle^{2k}).$$
Noticing that $\lnorm \phi_i\rnorm^2_2$ is a standard $\chi^2$ random variable with $d_i$ degrees of freedom, we know its moments are $\E(\lnorm \phi_i \rnorm_2^{2k})=2^k\Gamma\left(\frac{d_i}{2}+k\right)/\Gamma\left(\frac{d_i}{2}\right)$. Rearranging the previous display completes the proof.
\end{proof}

\begin{remark}
\label{rem:brauer}
    There is a proof of the complex claim \eqref{eq:symmetric_inj_norm_bounded} similar to the proof of the real claim. The details are a bit more involved due to the necessity of tracking complex conjugations arising in various Hermitian products. However, we did not find a proof of the real case mimicking the ideas of the complex case proof. In fact, $P_{\text{sym},k}^{(d_i)}$ is replaced by normalized sums over Brauer elements, which do not have the projector property. Although it is possible to obtain bounds on the injective norm from the sums over Brauer elements, we found those bounds are not as tight as the one obtained from the inductive proof. 
\end{remark}

\begin{remark}
    It is natural to wonder whether a bound like \eqref{eq:symmetric_inj_norm_bounded} holds in the real case. However, our proof of \eqref{eq:symmetric_inj_norm_bounded} mimics that of the complex case, which does not seem to naturally generalize to the real case (see Remark \ref{rem:brauer}); and our proof of the real bound \eqref{eq:real_inj_norm_bounded} depends strongly on the $u_i$'s being different, so that we can induct on them. Of course, since the bound \eqref{eq:real_inj_norm_bounded} applies to all real tensors, we can in particular apply it to real symmetric tensors. However, the resulting bound does not seem to be tight for our models.
    Even putting proofs aside, it is not clear what the correct bound would be. From pattern-matching in \eqref{eq:real_inj_norm_bounded}, \eqref{eq:inj_norm_bounded}, and \eqref{eq:symmetric_inj_norm_bounded}, one might guess that real-symmetric tensors $B \in \Sym_p(\R^d)$ satisfy the bound
    \begin{equation}
    \label{eqn:real-sym-guess}
        \injnormR{B}^{2k} \overset{?}{\leq} \frac{2^{pk}\Gamma(\frac{d}{2}+pk)}{(2pk-1)!!\Gamma(\frac{d}{2})} \E_u\left( \ip{B,u^{\otimes p}}^{2k} \right).
    \end{equation}
    However, this bound is not correct, even in the matrix case $p = 2$; if $B = \diag(b_1,\ldots,b_d)$ is a diagonal matrix, then $\ip{B,u^{\otimes p}} = \sum_{i=1}^d b_i u_i^2$, so that (from explicit moments of entries of uniform unit vectors) \eqref{eqn:real-sym-guess} becomes
    \[
        \max_{i=1}^d \abs{b_i} \overset{?}{\leq} \frac{2(d-1)}{3d} \sum_{i=1}^d b_i^2 + \frac{d+2}{3d} \left(\sum_{i=1}^d b_i \right)^2.
    \]
    This fails if, for example, $d = 2$, $b_1 = 1$, and $b_2 = -1/6$.
\end{remark}

\begin{remark}\label{rem:bounds-to-invariants}
    These bounds can be expressed as sums of local unitary invariant polynomials of tensors. In fact, in the non-symmetric case, from equation \eqref{eq:uniform-to-Gaussian}, the averages reduce to $$\E_{u}(\lvert \langle T,u_1\otimes\ldots\otimes u_p\rangle\rvert^{2k})=\frac1{\E_{\phi}(\lnorm \phi \rnorm_{2}^{2k})^p}\sum_{\sigma_1,\ldots,\sigma_p\in \mathfrak{S}_k}\langle T^{\otimes k}\vert \sigma_1\otimes \ldots\otimes \sigma_p\vert T^{\otimes k}\rangle,$$
     where $\phi$ is a standard normal complex $d$-dimensional random vector. For each choice of $\sigma_1,\ldots,\sigma_p\in\mathfrak{S}_k$,  $\langle T^{\otimes k}\vert \sigma_1\otimes \ldots\otimes \sigma_p\vert T^{\otimes k}\rangle$ is recognized to be the LU unitary invariant of $T$ constructed from $\sigma_1,\ldots,\sigma_p$ as has been described several times in the literature in \textit{e.g.} \cite[section 6, eq. (224)]{bonzom2018blobbed}, \cite[section 2.2]{collins2024tensor}, whose corresponding colored graph is the bipartite multi-graph $G_{\sigma_1,\ldots,\sigma_p}=(V=V_\circ\sqcup V_\bullet,E=\bigsqcup_{i=1}^pE_i)$, with white vertices $V_\circ=\{1_\circ,\ldots,k_\circ\},$ black vertices $V_\bullet=\{1_\bullet,\ldots,k_\bullet\}$ and edge multiset $E$ partitioned into $i$-colored edge sets $E_i,\ i=1\ldots,p$, $E_i=\{(v_\circ,\sigma_i(v)_\bullet):v_\circ\in V_\circ\}$. 
     In the symmetric case, similar considerations lead us to similar conclusions. The invariance is now under the action of $U\in U(d),$  $T\mapsto U^{\otimes p}T$.  The corresponding invariants have not been formally described in the literature. However, it is a simple exercise to show that those are generated by similar graphs contractions of tensors, with the corresponding graphs now being bipartite regular $p$-valent multi-graphs, without edge coloration. All the content of this remark easily generalizes to the real case, in which case the corresponding graphs need not be bipartite.
\end{remark}

\begin{remark}
\label{rem:latala}
We remark that, in the real asymmetric case, bounds retaining the spirit of Lemma \ref{lem:gaussian_deterministic_bound}  appeared in Lata\l a's well-known work \cite{Lat2006}, albeit with very different proofs. Lemma \ref{lem:gaussian_deterministic_bound} gives 
\begin{equation}
\label{eqn:ours-compared-to-latalas}
    \injnormR{T} \leq \left( ((2k-1)!!)^{\frac{1}{k}} \right)^{-\frac{p}{2}} \| \ip{T,\varphi_1 \otimes \cdots \otimes \varphi_p} \|_{2k},
\end{equation}
where $\|X\|_{2k} = \E[X^{2k}]^{1/(2k)}$ is the $L^{2k}$ norm of a random variable $X$. Since $(2k-1)!! = \frac{(2k)!}{2^kk!}$, the standard bounds $\sqrt{2\pi n}(n/e)^n \exp(\frac{1}{12n+1}) \leq n! \leq \sqrt{2\pi n}(n/e)^n \exp(\frac{1}{12n})$ yield that 
\[
    ((2k-1)!!)^{\frac{1}{k}} \geq \frac{2k}{e} \exp\left(\frac{1}{k} \left( \frac{\log 2}{2} + \frac{1}{24k+1} - \frac{1}{12k}\right) \right) \geq \frac{2k}{e},
\]
so that \eqref{eqn:ours-compared-to-latalas} gives 
\begin{equation}
\label{eqn:worse-ours-compared-to-latalas}
    \injnormR{T} \leq e^{p/2} (2k)^{-\frac{p}{2}} \| \ip{T,\varphi_1 \otimes \cdots \otimes \varphi_p} \|_{2k}.
\end{equation}
Stirling's $\lim_{k \to \infty} \frac{((2k-1)!!)^{\frac{1}{k}}}{k} = \frac{2}{e}$, so that \eqref{eqn:worse-ours-compared-to-latalas} is not particularly lossy for large $k$.

Comparatively, Theorem 1 of \cite{Lat2006} gives two-sided bounds on $\|\ip{T,\varphi_1 \otimes \cdots \otimes \varphi_p}\|_{2k}$ in terms of weighted sums of norms of flattenings of the deterministic tensor $T$. Since $\injnormR{T}$ corresponds to $\|T\|_{\{1\},\ldots,\{d\}}$ in Lata\l a's notation, his Theorem 1 gives the existence of some constants $C_p$ such that
\[
    (2k)^{p/2} \injnormR{T} \leq m_{2k}(T) \leq C_p \|\ip{T,\varphi_1 \otimes \cdots \otimes \varphi_p}\|_{2k}
\]
where $m_{2k}(T)$ is defined in \cite{Lat2006}, and thus
\begin{equation}
\label{eqn:latalas-compared-to-ours}
    \injnormR{T} \leq C_p (2k)^{-p/2} \|\ip{T,\varphi_1 \otimes \cdots \otimes \varphi_p}\|_{2k}.
\end{equation}
Comparing \eqref{eqn:worse-ours-compared-to-latalas} and \eqref{eqn:latalas-compared-to-ours}, one sees that our result is akin to giving an explicit formula for $C_p$ in this consequence of Lata\l a's result. (Since we have dropped the other flattenings of $T$, our does not give an explicit formula for the constants in Lata\l a's Theorem 1.) One of the primary interests of our paper is that the constants we obtain are conjecturally tight for random tensors; this will be discussed at length below.
\end{remark}

\begin{remark}
The rest of this paper will discuss the quality of these bounds for $p \geq 2$; here, we just remark that in the simple case $p = 1$ of vectors, the bounds of Theorem \ref{thm:deterministic_bound} are actually equalities for every deterministic $T$ and every $k$. The basic argument is that the injective norm of such a vector is its $2$-norm; one can first rewrite $\E[|\ip{T,u}|^{2k}]$ by relating the uniform vector $u$ to a Gaussian vector $\varphi$ of the correct symmetry class; then, from rotational symmetry of the Gaussian, $\ip{T,\varphi}$ is itself a Gaussian variable, so all moments can be computed. In the real case, one can do this concretely by combining \eqref{eqn:injective-norm-of-vectors} and \eqref{eqn:uniform-to-gaussian-real-case}. In the complex case, one just replaces real Gaussian vectors with complex Gaussian vectors; for example, \eqref{eqn:injective-norm-of-vectors} is replaced with $\E[|\ip{T,\varphi_1}|^{2k}] = k! \|T\|_2^{2k}$, where $\varphi_1$ is a standard complex Gaussian, which has moments $\E[\|\varphi_1\|_2^{2k}] = \Gamma(d+k)/\Gamma(d)$. 
\end{remark}


\section{Application: Asymmetric tensors} 
\label{sec:application_asymmetric}

We let $\psi_p^\K:\R_+\times (\R_+)^{p-1}\rightarrow \R$ be the function:
\begin{equation}
    \psi_p^\K(\alpha;\eta_2,\ldots,\eta_p):=\begin{cases}\frac1{\sqrt{e}(\eta_2\ldots\eta_p)^{\frac1{2p}}}\left(\frac{(1+\alpha)^{1+\alpha}}{\alpha^{1+\alpha}}\right)^{1/2} \left( \prod_{i=2}^p\frac{(1+\alpha \eta_i)^{1+\alpha \eta_i}}{(\alpha \eta_i)^{\alpha \eta_i}}\right)^{1/2}, \text{ if } \K=\C, \\
    \frac{1}{\sqrt{e}(\eta_2\ldots \eta_p)^{\frac1{2p}}}\left( \frac{(1+\alpha/2)^{1+\alpha/2}}{(\alpha/2)^{1+\alpha/2}}\right)^{1/2}\left(\prod_{i=2}^p\frac{(1+\alpha\eta_i/2)^{1+\alpha\eta_i/2}}{(\alpha\eta_i/2)^{\alpha\eta_i/2}}\right)^{1/2}, \text{ if } \K=\R.
    \end{cases}
\end{equation}
As a short version of the above notation we denote $\psi_p^\K(\alpha)=\psi_p^\K(\alpha;1,\ldots,1).$ Note that 
\[
    \psi_p^\C(\alpha;\eta_2,\ldots,\eta_p)=\psi_p^\R(2\alpha;\eta_2,\ldots,\eta_p).
\]
\begin{theorem}
\label{thm:asymmetric}
If $T$ comes from \hyperlink{model:ak}{Model $A_\K$}, then we have the following upper bounds for its injective norm in various asymptotic regimes:
\begin{itemize}
\item Suppose that $d_1 = \cdots = d_p = d$ is fixed. Then as $p \to \infty$ we have
\begin{equation}
    \limsup_{p \to \infty} \frac{1}{\sqrt{p \log p}} \E[\injnormK{T}] \leq \sqrt{\frac{d-1}{d}},
\end{equation}
more precisely with the following rate: For all $\epsilon > 0$ there exists $p_0$ such that for all $p \geq p_0$ we have
\begin{equation}
\label{eqn:AR-p-to-infinity-rate}
    \frac{1}{\sqrt{p \log p}} \E[\injnormK{T}] \leq \sqrt{\frac{d-1}{d}}\left(1 + \left( \frac{1}{2}+\epsilon \right) \frac{\log \log p}{\log p} \right).
\end{equation}
\item Suppose $p\ge 2$ is fixed, and $d_1,\ldots,d_p \to \infty$ in such a way that $d_i = \lfloor \eta_i d_1 \rfloor$ for some fixed $\eta_2,\ldots,\eta_p > 0$. Then
    \begin{equation}\label{eq:multivariate_bound}
        \limsup_{d_1\to \infty}\E[\injnormK{T}] \le \inf_{\alpha\in \R_+} \psi_p^\K(\alpha;\eta_2,\ldots,\eta_p).
    \end{equation}
    In particular, for $d_1=\ldots=d_p$ all sent to infinity, one has $\limsup_{d_1\to \infty}\E\injnormK{T}\le \inf_{\alpha\in \R_+^*}\psi_p^\K(\alpha),$ and the infimum is realized uniquely at the unique positive solution $\alpha^\K_0(p)$ to 
    \begin{equation}
        \begin{cases}
            \alpha_0^\C(p) \log(1+1/\alpha_0^\C(p))=1/p, \text{ if } \K=\C,\\
            \alpha_0^\R(p) \log(1+2/\alpha_0^\R(p))=2/p, \text{ if } \K=\R,
        \end{cases}
    \end{equation}
    which (clearly) satisfy $\alpha_0^\R(p) = 2\alpha_0^\C(p)$.
    Moreover, 
    \begin{equation}
    \label{eqn:inj-norm-limit-d-then-p}
        \limsup_{d\to\infty}\E[\injnormK{T}] \le\psi_p^\K(\alpha_0^\K(p))\le\sqrt{p\log p}(1+\oo_{p\to\infty}(1)).
    \end{equation}
\end{itemize}
\end{theorem}

\begin{remark}
\label{rem:kac--rice}
The \emph{Gaussian} special cases of equations \eqref{eqn:AR-p-to-infinity-rate} and \eqref{eqn:inj-norm-limit-d-then-p} were essentially previously known, from our prior work \cite{dartois2024injective}, even with essentially the same rate $\frac{\log\log p}{2\log p}$. (The prior work was based on very different arguments, namely spin-glass methods and the Kac--Rice formula; see Appendix \ref{app:prior-bounds-comparison} for more discussion of the methods. The prior work also gave high-probability bounds, rather than bounds in expectation, but the bounds were the same.) However, the non-Gaussian cases are new to the best of our knowledge.

In the \emph{real Gaussian} case, $p$ fixed, a matching lower bound to \eqref{eqn:inj-norm-limit-d-then-p} is known from \cite{bates2025balanced} (albeit with high probability, rather than in expectation). However, in the \emph{complex Gaussian} case, we are not aware of any matching lower bound to \eqref{eqn:inj-norm-limit-d-then-p}.
\end{remark}
The proof relies on the next lemma, which permits one to bound the average injective norm of $\K$-rigidly sub-Gaussian tensors using their Gaussian counterpart.
\begin{lemma}\label{lem:rigid-subgaussian-moment-bound}
Let $\{u_i\in \K^{d_i}\}_{i=1}^p$ be a family of deterministic unit-norm vectors, \textit{i.e.} $\lnorm u_i\rnorm_2=1$. Then, if $T\in \bigotimes_{i=1}^p\K^{d_i}$ comes from \hyperlink{model:ak}{Model $A_\K$}, we have
\begin{equation}\label{eqn:rigid-subgaussian-moment-bound}
        \E_{T}(|\ip{T,u_1 \otimes \cdots \otimes u_p}|^{2k}) \leq \delta_{\K,\C}\frac{k!}{(d_1\cdots d_p)^{k/p}}+\delta_{\K,\R}\frac{(2k-1)!!}{(d_1\ldots d_p)^{k/p}}.
    \end{equation}
\end{lemma}

\begin{proof}We split the proof in two cases, $\K=\R$ or $\K=\C$.\\

\noindent{\it Real case:} \\
 Let $W\in \bigotimes_{i=1}^p \R^{d_i}$ be a real random tensor with i.i.d. $\mathcal{N}_\R\left(0,\frac1{(d_1\ldots d_p)^{1/p}}\right)$ entries.
We have 
$$\E_W(\langle W,u_1\otimes \ldots\otimes u_p\rangle^{2k})=\frac{(2k-1)!!}{(d_1\ldots d_p)^{k/p}}\left(\prod_{i=1}^p\langle u_i,u_i\rangle^{k}\right)=\frac{(2k-1)!!}{(d_1\ldots d_p)^{k/p}}.$$ 
Assume now that $T\in \bigotimes_{i=1}^p \R^{d_i}$ is a real random tensor of \hyperlink{model:ak}{Model $A_\R$}. One has 
\begin{equation*}
 \E_{T}(\langle T,u_1\otimes \ldots\otimes u_p\rangle^{2k})=\sum_{I_1,\ldots, I_{2k}\in \prod_{i=1}^p[d_i]}\left(\prod_{s=1}^{2k}(u_1\otimes \ldots \otimes u_p)_{I_s}\right)\E_{T}\left( \prod_{s=1}^{2k}T_{I_s}\right),
\end{equation*}
where we shorten notations by introducing multi-indices $I_s=(i_1^{(s)},\ldots, i_p^{(s)})\in \prod_{i=1}^p[d_i].$ In order to manipulate the above sum, let us introduce $\mathcal{P}_{\text{even}}([2k])$ the set of even partitions of the set $[2k]$, where by even we mean that the blocks $b$ of any partition $\pi \in \mathcal{P}_{\text{even}}([2k])$ must be of even size $\lvert b\rvert\in 2\mathbb{N}^*.$ Using independence and symmetry of the entries of $T$ we reach
\begin{multline*}
    \sum_{I_1,\ldots, I_{2k}\in \prod_{i=1}^p[d_i]}\left(\prod_{s=1}^{2k}(u_1\otimes \ldots \otimes u_p)_{I_s}\right)\E_{T}\left( \prod_{s=1}^{2k}T_{I_s}\right)=\\
    \sum_{\pi \in \mathcal{P}_{\text{even}([2k])}}\sum_{\substack{I_b: b\in \pi\\I_b\neq I_{b'}, b\neq b'}}\left(\prod_{b\in \pi} (u_1\otimes \ldots \otimes u_p)^{\lvert b \rvert}_{I_b}\right)\prod_{b\in \pi}\E_T\left(T^{\lvert b \rvert}_{I_b}\right).
\end{multline*}
Our $\R$-rigid sub-Gaussianity assumption here reads $\E_T\left(T^{\lvert b \rvert}_{I_b}\right)\le \E_W\left( W_{I_b}^{\lvert b\rvert} \right)$, which implies
\begin{multline*}
     \E_{T}(\langle T,u_1\otimes \ldots\otimes u_p\rangle^{2k})\le \sum_{\pi \in \mathcal{P}_{\text{even}([2k])}}\sum_{\substack{I_b: b\in \pi\\I_b\neq I_{b'}}}\left(\prod_{b\in \pi} (u_1\otimes \ldots \otimes u_p)^{\lvert b \rvert}_{I_b}\right)\prod_{b\in \pi}\E_W\left(W^{\lvert b \rvert}_{I_b}\right)\\
     =\E_{W}\left( \langle W, u_1\otimes \ldots \otimes u_p\rangle^{2k}\right).
\end{multline*}
This proves the real claim.\\

\noindent{\it Complex case:}\\
  The structure of the argument is similar. We here record the small changes needed. We repeat earlier steps by introducing $Z\in \bigotimes_{i=1}^p \C^{d_i}$ a random tensor with i.i.d. $\mathcal{N}_\C\left(0,\frac1{(d_1\ldots d_p)^{1/p}}\right)$ entries. 
     For any choice of complex unit norm vectors $u_1,\ldots,u_p$ we have 
    $$\E_Z\left( \lvert \langle Z,u_1\otimes \ldots\otimes u_p\rangle\rvert^{2k}\right)=\frac{k!}{(d_1\ldots d_p)^{k/p}}\left( \prod_{i=1}^p\lvert \langle u_i,u_i\rangle\rvert^{k}\right)=\frac{k!}{(d_1\ldots d_p)^{k/p}}.$$ Assume now that $T$ is a random tensor satisfying the assumptions of \hyperlink{model:ak}{Model $A_\C$}. One has
\begin{multline*}
\E_{T}\left( \lvert \langle T,u_1\otimes\ldots \otimes u_p\rangle\right\rvert^{2k})=\\
\sum_{\substack{I_1,\ldots,I_k\in \prod_{i=1}^p[d_i]\\J_1,\ldots,J_k\in \prod_{i=1}^p[d_i]} }\left(\prod_{s=1}^k(u_1\otimes\ldots \otimes u_p)_{I_s}(\bar{u}_1\otimes\ldots \otimes \bar{u}_p)_{J_s}\right)\E_T\left( \prod_{s=1}^kT_{I_s} \overline{T}_{J_s}\right),
\end{multline*}
where we shorten notations by using multi-indices $I_s,J_s=(i_1^{(s)},\ldots,i_p^{(s)}),(j_1^{(s)},\ldots,j_p^{(s)})\in [d_1]\times \ldots \times [d_p].$ We split $\E_T\left( \prod_{s=1}^kT_{I_s} \overline{T}_{J_s}\right)$ over independent components. To this aim let us introduce $\mathcal{P}([k]\sqcup[\bar k])$ the set of partitions of the set $[k]\sqcup[\bar k]:=\{1,\ldots,k\}\sqcup \{\bar 1,\ldots, \bar k\}$. The set $[\bar k]$ is just a copy of $[k]$ but with different symbols denoting its elements. We introduce the set of balanced partitions $\mathcal{P}_{\text{bal}}([k]\sqcup[\bar k])\subset\mathcal{P}([k]\sqcup[\bar k])$ as the set of partitions $\pi$ of $[k]\sqcup[\bar k]$  whose blocks $b\in \pi$ are such that $\lvert b\cap [k]\rvert=\lvert b\cap [\bar k]\rvert$ (so that in particular $\lvert b \rvert=2q$ for some $q\in \N$).  We can now further write
\begin{multline*}
    \E_{T}\left( \lvert \langle T,u_1\otimes\ldots \otimes u_p\rangle\right\rvert^{2k})=\\
    \sum_{\pi \in \mathcal{P}_{\text bal}([k]\sqcup[\bar k])}
    \sum_{\substack{I_b\in \prod_{i=1}^p[d_i]^p: b\in \pi\\ I_{b}\neq I_{b'}, b\neq b'}}\left(\prod_{b\in \pi}\lvert(u_1\otimes\ldots \otimes u_p)_{I_b}\rvert^{\lvert b \lvert}\right)
    \prod_{b\in \pi}\E_{T}\left(\lvert T_{I_b}\rvert^{\lvert b\rvert}\right),
\end{multline*}
which follows from independence and circular symmetry of the entries. As a consequence of $\C$-rigid sub-Gaussianity, one has $\E_{T}\left(\lvert T_{I_b}\rvert^{\lvert b\rvert}\right)\le \E_{Z}\left(\lvert Z_{I_b}\rvert^{\lvert b\rvert}\right),$ so that
\begin{multline*}
    \E_{T}\left( \lvert \langle T,u_1\otimes\ldots \otimes u_p\rangle\right\rvert^{2k})\le\\
    \sum_{\pi \in \mathcal{P}_{\text{bal}}([k]\sqcup[\bar k])}
    \sum_{\substack{I_b\in \prod_{i=1}^p[d_i]^p: b\in \pi\\ I_{b}\neq I_{b'}, b\neq b'}}\left(\prod_{b\in \pi}\lvert(u_1\otimes\ldots \otimes u_p)_{I_b}\rvert^{\lvert b \lvert}\right)
    \prod_{b\in \pi}\E_{Z}\left(\lvert Z_{I_b}\rvert^{\lvert b\rvert}\right)\\
    =\E_{Z}\left( \lvert \langle Z,u_1\otimes\ldots \otimes u_p\rangle\right\rvert^{2k}).
\end{multline*}
This proves the complex claim.
\end{proof}

\begin{proof}[Proof of Theorem \ref{thm:asymmetric}]

\noindent{\it \RomanNum{1}. $d_1,\ldots,d_p=d$ fixed and $p\to \infty.$} \\

\noindent{\it Complex case:} From equations \eqref{eq:inj_norm_bounded} and \eqref{eqn:rigid-subgaussian-moment-bound}, we have the estimate
\begin{equation}
\label{eqn:p-to-infinity-complex-calculus-overall}
    \sqrt{d}\E[\injnormC{T}] \leq \sqrt{d}\E[\injnormC{T}^{2k}]^{\frac{1}{2k}} \leq \left(k!\binom{d+k-1}{k}^p \right)^{\frac{1}{2k}}
\end{equation}
for any positive integer $k$. The basic idea is to take $k = \alpha p\log p$ for $\alpha > 0$, then use Stirling's formula to find asymptotics and optimize in $\alpha$; the optimizer turns out to be $\alpha = d-1$, and lower-order Taylor corrections give the rate \eqref{eqn:AR-p-to-infinity-rate}. Since $k$ must be an integer, we actually use $k = \lfloor \alpha p \log p \rfloor$.

In more details: A Stirling's computation gives
\begin{equation}
\label{eqn:p-to-infinity-complex-calculus-first-stirlings}
    \lim_{k \to \infty} \left(\frac{(k!)^{\frac{1}{2k}}}{\sqrt{k}} \cdot \sqrt{e} - 1\right)\frac{k}{\log k} = \frac{1}{4},
\end{equation}
so that, for any $\epsilon > 0$ and large enough $k$ (therefore, large enough $p$) we have that
\[
    (k!)^{\frac{1}{2k}} \leq \sqrt{\frac{k}{e}} \left(1 + \left(\frac{1}{4}+\epsilon\right)\frac{\log k}{k}\right).
\]
With the choice $k = \lfloor \alpha p \log p \rfloor$, we have $\frac{\log k}{k} \leq \frac{\log(\alpha p \log p)}{\alpha p \log p - 1}$. Since, for any fixed $\alpha > 0$, we have $\lim_{p \to \infty} \frac{\log(\alpha p \log p)}{\alpha p \log p - 1} \cdot \alpha p = 1$, for any $\epsilon > 0$ and $p$ large enough we have we have $\frac{\log k}{k} \leq \frac{1+\epsilon}{\alpha p}$. Combining these and absorbing the $\epsilon$, we find that for any $\alpha, \epsilon > 0$ and $p$ large enough, we have
\[
    (k!)^{\frac{1}{2k}} \leq \sqrt{\frac{\alpha p \log p}{e}}\left(1 + \frac{1/4+\epsilon}{\alpha p}\right). 
\]
If $k = \lfloor \alpha p \log p \rfloor$ and $\beta$ is any fixed positive integer, a bound using $(k+1)^\beta \leq \frac{(\beta+k)!}{k!} \leq (k+\beta)^\beta$ gives
\begin{equation}
\label{eqn:p-to-infinity-complex-calculus-second-stirlings}
    \lim_{p \to \infty} \left( \frac{\binom{\beta + k}{k}^\frac{p}{2k}}{\exp\left(\frac{\beta}{2\alpha}\right)} - 1 \right)\frac{\log p}{\log \log p} = \frac{\beta}{2\alpha}.
\end{equation}
In other words, for any $\epsilon > 0$ we have
\[
    \binom{d+k-1}{k}^{\frac{p}{2k}} \leq \exp\left(\frac{\beta}{2\alpha}\right)\left( 1 + (1+\epsilon)\frac{\beta}{2\alpha} \cdot \frac{\log \log p}{\log p}\right).
\]
Combining \eqref{eqn:p-to-infinity-complex-calculus-overall}, \eqref{eqn:p-to-infinity-complex-calculus-first-stirlings}, and \eqref{eqn:p-to-infinity-complex-calculus-second-stirlings}, we find that for any $\epsilon > 0$ we have
\[
    \sqrt{\frac{d}{p \log p}} \cdot \E[\injnormC{T}] \leq \exp\left(\frac{d-1}{2\alpha} + \frac{\log(\alpha) - 1}{2}\right) \left( 1 + \frac{1/4+\epsilon}{\alpha p}\right) \left(1+(1+\epsilon)\frac{d-1}{2\alpha} \cdot \frac{\log \log p}{\log p}\right).
\]
for all $p$ large enough. The function $\alpha \mapsto \frac{d-1}{2\alpha} + \frac{\log(\alpha)-1}{2}$ is minimized at $\alpha = d-1$, where it takes the value $\log(d-1)/2$. The dominant term in the correction is $1 + (1+\epsilon)\frac{d-1}{2\alpha} \frac{\log \log p}{\log p}$, and for this choice of $\alpha$ we have $\frac{d-1}{2\alpha} = \frac{1}{2}$, which finishes the proof.

\noindent{\it Real case:} Analogously to the complex case, we start by using equations \eqref{eq:real_inj_norm_bounded} and \eqref{eqn:rigid-subgaussian-moment-bound} to obtain
\begin{equation}
\label{eqn:real-case-finite-bound}
    \sqrt{d} \E[\injnormR{T}] \leq \sqrt{d} \E[\injnormR{T}^{2k}]^{\frac{1}{2k}} \leq \left( (2k-1)!! \left( \frac{2^k \Gamma(\frac{d}{2}+k)}{(2k-1)!! \Gamma(\frac{d}{2})} \right)^p \right)^{\frac{1}{2k}}
\end{equation}
for any positive integer $k$. The proof is essentially exactly the same as before; we replace \eqref{eqn:p-to-infinity-complex-calculus-first-stirlings} with 
\[
    \lim_{k \to \infty} \left( \frac{((2k-1)!!)^{\frac{1}{2k}}}{\sqrt{k}} \cdot \sqrt{\frac{e}{2}} - 1\right) k = \frac{\log 2}{4},
\]
replace \eqref{eqn:p-to-infinity-complex-calculus-second-stirlings} with 
\[
    \lim_{p \to \infty} \left(\frac{\left(\frac{2^k\Gamma(\beta + k)}{(2k-1)!!\Gamma(\beta)}\right)^{\frac{p}{2k}}}{\exp\left(\frac{2\beta-1}{4\alpha}\right)} - 1 \right) \frac{\log p}{\log \log p} = \frac{2\beta-1}{4\alpha},
\]
and use these to obtain that, for every $\epsilon > 0$,
\[
    \sqrt{\frac{d}{p \log p}} \E[\injnormR{T}] \leq \exp\left(\frac{d-1}{4\alpha} + \frac{\log(2\alpha) - 1}{2} \right) \left( 1 + \frac{\frac{\log 2}{4} + \epsilon}{\alpha p \log p} \right) \left(1 + (1+\epsilon) \frac{d-1}{4\alpha} \cdot \frac{\log \log p}{\log p}\right)
\]
for all $p$ large enough. The optimizer is $\alpha = (d-1)/2$. \\

\noindent{\it \RomanNum{2}. $d_i\to \infty, p$ fixed regime.} \\
\noindent{\it Real case:}\\
Let $T$ be a random tensor of \hyperlink{model:ak}{Model $A_\R$}. Arguing as in \eqref{eqn:real-case-finite-bound}, we obtain
\begin{equation*}
   \E_T\injnormR{T}\le\left[2^{pk}\left(\prod_{i=1}^p\frac{\Gamma(d_i/2+k)}{(2k-1)!!\Gamma(d_i/2)}\right)\frac{(2k-1)!!}{(d_1\ldots d_p)^{k/p}}\right]^{\frac1{2k}} 
\end{equation*}
for all positive integers $k$. Let $d_i$ and $\eta_i$ be as in the theorem statement. Fix some $\alpha > 0$, and let $k = \lfloor d_1/\alpha \rfloor$. Then it follows that
\begin{equation*}
    \lim_{k\to \infty} \left[2^{pk}\left(\prod_{i=1}^p\frac{\Gamma(d_i/2+k)}{(2k-1)!!\Gamma(d_i/2)}\right)\frac{(2k-1)!!}{(d_1\ldots d_p)^{k/p}}\right]^{\frac1{2k}}= \psi^{\R}_p(\alpha; \eta_2,\ldots,\eta_p),
\end{equation*}
since for all $a>0$
\begin{equation*}
    \lim_{k\to \infty} \left( \frac{(2k-1)!!}{\lfloor a k\rfloor^k}\right)^{1/2k}=\sqrt{\frac{2}{a e}}, \quad  \lim_{k\to \infty}\left(\frac{\Gamma(\lfloor a k\rfloor/2+k)}{(2k-1)!!\Gamma(\lfloor a k \rfloor/2)} \right)^{1/2k}=\frac12 a^{-a/4}(2+a)^{\frac{2+a}{4}}.
\end{equation*}
Hence, for all $\alpha > 0$ we have
\begin{equation*}
   \limsup_{d_1\to \infty}\E_T\injnormR{T}\le\psi^{\R}_p(\alpha; \eta_2,\ldots,\eta_p),
\end{equation*}
and minimizing over $\alpha$ leads to \eqref{eq:multivariate_bound} for $\K=\R$. 
We now prove the uniqueness of the minimizer when $d_1=d_2=\ldots=d_p=d$. To this aim we compute the derivative
\[
    \psi_p^{\R}{}'(\alpha)=\sqrt{\frac1{e 2^{p+3}\alpha^3}} (\alpha^{-\alpha/4}(2+\alpha)^{\frac{2+\alpha}{4}})^p\left(p\alpha\log\left(1+2/\alpha\right)-2\right).
\]
Note that $\psi_p^{\R}{}'(\alpha)$ is of the same sign as $f_p(\alpha)=p\alpha\log\left(1+2/\alpha\right)-2$. One easily shows $f_p'(\alpha)>0$ for all $\alpha>0$ (compute the limiting values of $f_p'$ at $0,+\infty$ and study the sign of $f_p''(\alpha)$). This shows $f_p(\alpha)$ is strictly increasing. Moreover, $\lim_{\alpha \to 0^+}f(\alpha)=-2$ and $\lim_{\alpha\to \infty}f(\alpha)= 2(p-1)$, hence $f_p(\alpha)$ changes sign once and so does $\psi_p^{\R}{}'(\alpha)$. Note that $\lim_{\alpha\to 0^+}\psi_p^{\R}(\alpha)=+\infty$ and $\lim_{\alpha\to \infty}\psi_p^{\R}(\alpha)=+\infty$ which suffices for the uniqueness of the minimum $\alpha_0^\R(p)$ of $\psi_p^{\R}$.\\

The asymptotic claim for large $p$ is obtained by comparing the value of $\psi_p^\R(\alpha)$ at the solution of $f_p(\alpha)=0$ with its value at the solution of a simpler equation solved by Lambert function. Let $y:=\frac{2}{p\alpha}$; then $f_p(\alpha)=0$ is equivalent to $\log(1+py)=y$, i.e., is equivalent to $\log(py)+\log(1+\frac1{py})=y.$ 
Let $y'$ be the largest solution of $\log(py')=y'$. Since for all $x>0, \ \log(1+px)>\log(px),$ one must have $y\neq y'$. Let $p\ge 3$; then $y'$ satisfies $-y'e^{-y'}=-\frac1{p}.$ The largest solution $y'$ is expressed using the second branch of the Lambert function $y'=-\mathcal{W}_{-1}(-1/p)$. The asymptotic expansion as $p\to \infty$ is known from the expansion for small arguments of $\mathcal{W}_{-1}$, see \cite[pp. 349-350]{corless1996lambert}. We recall,
$$\mathcal{W}_{-1}(x)=\log(-x)-\log(-\log(-x))+\frac{\log(-\log(-x))}{\log(-x)}+O_{x\to 0^-}\left(\left( \frac{\log(-\log(-x))}{\log(-x)}\right)^2\right).$$
Thus
$$y'=\log(p)+\log\log(p)+\frac{\log\log(p)}{\log(p)}+O_{p\to \infty}\left(\left(\frac{\log\log(p)}{\log(p)}\right)^2\right).$$
Let $\alpha'=\frac2{py'}$, $\alpha'= \frac2{p\log p}(1+o_p(1))$. Since $\psi_p^\R(\alpha)$ is uniquely minimized at $\alpha_0^\R(p),$  $$\psi_p^\R(\alpha^\R_0(p))\le\psi_p^\R(\alpha'),$$
while one easily shows $\psi_p^\R(\alpha')=\sqrt{p\log p}(1+o_p(1))$ using $\alpha'=\frac2{p\log p}(1+o_p(1)).$ \\

\noindent{\it Complex case:}\\
Let $T$ be a random tensor of \hyperlink{model:ak}{Model $A_\C$}. The proof technique of the real case above applies here up to minor changes. One starts instead with equation \eqref{eq:inj_norm_bounded} of Theorem \ref{thm:deterministic_bound} and \eqref{eqn:rigid-subgaussian-moment-bound} of Lemma \ref{lem:rigid-subgaussian-moment-bound}. This leads to
\begin{align*}
    \E_T\injnorm{T}&\le \left( \prod_{i=1}^p\binom{d_i+k-1}{k} \E_Z\E_{u_i}\left( \lvert \langle Z,u_1\otimes \ldots\otimes u_p\rangle\rvert^{2k}\right) \right)^{\frac1{2k}}\\
    &\le\left(\frac{k!}{(d_1\ldots d_p)^{k/p}}\prod_{i=1}^p\binom{d_i+k-1}{k}\right)^{\frac1{2k}}.
\end{align*}
Letting for all $i=2,\ldots, p$, $\eta_i>0$, $d_i=\lfloor \eta_i d_1\rfloor$ and $d_1=\lfloor \alpha k\rfloor$ then computing the limit of the right hand side above, we find
\begin{multline*}
    \lim_{k\to\infty}\left(\frac{k!}{(d_1\ldots d_p)^{k/p}}\prod_{i=1}^p\binom{d_i+k-1}{k}\right)^{\frac1{2k}}\\
    =\frac1{\sqrt{e}(\eta_2\ldots\eta_p)^{\frac1{2p}}}\left(\frac{(1+\alpha)^{1+\alpha}}{\alpha^{1+\alpha}}\right)^{1/2} \left( \prod_{i=2}^p\frac{(1+\alpha \eta_i)^{1+\alpha \eta_i}}{(\alpha \eta_i)^{\alpha \eta_i}}\right)^{1/2},
\end{multline*}
which we minimize over $\alpha$ to get the result. The asymptotic large $p$ claim follows from the real case and the earlier remark $\psi_p^\C(\alpha;\eta_2,\ldots,\eta_p)=\psi_p^\R(2\alpha;\eta_2,\ldots,\eta_p)$.
\end{proof}

\begin{remark}
A  benefit of the large-moment bound in Theorem~\ref{thm:asymmetric} is that it extends with minimal changes to rectangular tensors in $\K^{d_1}\otimes \cdots \otimes \K^{d_p}$. In contrast, an approach based on counting critical points via the Kac--Rice formula becomes significantly more delicate outside the cubic setting $d_1 = \cdots = d_p$. An important obstacle is that the Hessian typically does not admit a tractable explicit limiting law, which seems to obstruct a direct evaluation of the generalized logarithmic complexity $\Sigma_p$ appearing in \cite[Lemma~2.7]{dartois2024injective}. 
\end{remark}


\section{Application: Symmetric tensors}
\label{sec:application_symmetric}

This section is devoted to random symmetric tensors of Models \hyperlink{model:sc}{$S_\C$} and \hyperlink{model:stildec}{$\widetilde{S}_\C$}. While real versions of random symmetric tensors are fundamental to spherical spin glass theory \cite{auffinger2013random}, their complex versions are representative of random bosonic quantum states. In this paper, due to the limitations of our deterministic bound \eqref{eq:symmetric_inj_norm_bounded}, we limit ourselves to the complex case. Before we start, let us state the following remark on the opportunity of introducing the two Models \hyperlink{model:sc}{$S_\C$} and \hyperlink{model:stildec}{$\widetilde{S}_\C$} of symmetric random tensors.
\begin{remark}\label{rem:SCvstildeSC}
    The two families \hyperlink{model:sc}{$S_\C$} and \hyperlink{model:stildec}{$\widetilde{S}_\C$} represent two natural ways of building random symmetric tensors with independent entries up-to-symmetry. Those two families have a nontrivial intersection. Indeed, if $\gamma$ (introduced in \hyperlink{model:stildec}{Model $S_\C$}) is a Gaussian random variable, and if $T$ (used for the definition of $\widetilde{B}$ in \hyperlink{model:stildec}{Model $\widetilde{S}_\C$}) is a Gaussian random tensor with i.i.d. entries distributed as $\gamma$, then $B\overset{\text{law}}{=}\widetilde{B}$. 
    However, neither example is included in the other. For example, if $B$ comes from \hyperlink{model:sc}{Model $S_\C$}, then all of its entries have the same law up to rescaling; if $\widetilde{B}$ comes from \hyperlink{model:stildec}{Model $\widetilde{S}_\C$}, then for example $B_{1,\cdots,1,2}$ has the distribution of the sum of two rescaled independent copies of some distribution $\gamma$, while $B_{1,\ldots,1}$ has the distribution of some rescaling of $\gamma$. Notice that the sum of two independent rescaled Steinhaus variables is not a rescaled Steinhaus variable\footnote{More generally, this would hold true if ``Steinhaus variable'' were replaced by any non-stable random variable.} (it is no longer supported on a circle, since the phases are independent). Thus, if we symmetrize a Steinhaus random state, we obtain a tensor from \hyperlink{model:stildec}{Model $\widetilde{S}_\C$} which cannot be obtained from \hyperlink{model:sc}{Model $S_\C$}, since its $(1,\ldots,1)$ and $(1,\ldots,1,2)$ entries are not the same distribution up to rescaling. Similarly, if we take $\gamma$ to be Steinhaus in the definition of \hyperlink{model:sc}{Model $S_\C$}, we obtain a tensor from \hyperlink{model:sc}{Model $S_\C$} which cannot be obtained from \hyperlink{model:stildec}{Model $\widetilde{S}_\C$}, since it if were, the underlying $T$ entries would have to be Steinhaus up to scaling (as can be seen from the $(1,\ldots,1)$ entry), but in this case the $(1,\ldots,1,2)$ entry would no longer be rescaled Steinhaus.
\end{remark}

Introducing first the function
\begin{equation}
    \phi_p:\R_+\to\R, \quad \alpha \mapsto \frac1{p^{p/2}\sqrt{e}}\alpha^{-\frac12(\alpha+1)}(p+\alpha)^{\frac12(p+\alpha)}
\end{equation}
we can prove the following theorem for symmetric random tensors of Models \hyperlink{model:sc}{$S_\C$} and \hyperlink{model:stildec}{$\widetilde{S}_\C$}.
\begin{theorem}\label{thm:symm_complex_tensors_bound}
    If $B$ comes from either Model \hyperlink{model:sc}{$S_\C$} or \hyperlink{model:stildec}{$\widetilde{S}_\C$}, then we have the following upper bounds for its injective norm in various regimes:
    \begin{itemize}
    \item If $d$ is fixed and $p \to \infty$, we have
    \[
        \limsup_{p \to \infty} \frac{1}{\sqrt{\log p}} \E[\injnormC{B}] \leq \sqrt{\frac{d-1}{d}},
    \]
    more precisely with the following rate: For all $\epsilon > 0$ there exists $p_0$ such that for all $p \geq p_0$ we have
    \[
        \frac{1}{\sqrt{\log p}} \E[\injnormC{B}] \leq \sqrt{\frac{d-1}{d}} \left( 1 + (1+\epsilon) \cdot \frac{2d-1}{4d-4} \cdot \frac{\log \log p}{\log p} \right).
    \]
    \item If $p$ is fixed and $d \to \infty$, we have 
    \begin{equation*}
        \limsup_{d\to\infty}\E[\injnormC{B}] \le \inf_{\alpha \in \R_+}\phi_p(\alpha).
    \end{equation*}
    The infimum is realized at some unique $\alpha_0^{S}(p)$ and one has $\phi_p(\alpha_0^{S}(p))\le\sqrt{\log p}(1+\oo_{p\to\infty}(1)).$
    \end{itemize}
\end{theorem}
As in Section \ref{sec:application_asymmetric}, we need the analog of Lemma \ref{lem:rigid-subgaussian-moment-bound}.
\begin{lemma}\label{lem:sym-rigid-subgaussian-moment-bound}
    Let $u\in \C^d$ be a unit norm vector. Let $B\in \Sym_{p}(\C^d)$ be a random symmetric tensor of \hyperlink{model:sc}{Model $S_\C$}. Then,
    \begin{equation}\label{eq:symmetric:rigid-sub-Gaussian-bound}
        \E_{B}\left(\lvert\langle B,u^{\otimes p}\rangle \rvert^{2k} \right)\le \frac{k!}{d^k}.
    \end{equation}
   Similarly, if $\widetilde{B}\in \Sym_p(\C^d)$ is a random tensor of \hyperlink{model:stildec}{Model $\widetilde{S}_\C$}, then
   \begin{equation}\label{eq:symmetric-tilde:rigid-sub-Gaussian-bound}
        \E_{\widetilde{B}}\left(\lvert\langle \widetilde{B},u^{\otimes p}\rangle \rvert^{2k} \right)\le \frac{k!}{d^k}.
    \end{equation}
\end{lemma}
\begin{proof}
    The proof of inequality \eqref{eq:symmetric:rigid-sub-Gaussian-bound} comes from an approach similar to the one detailed in the proof of Lemma \ref{lem:rigid-subgaussian-moment-bound}. Therefore, we only recall the essential ideas here.  Let $G\in \Sym_{p}(\C^d)$ be a complex Gaussian symmetric tensor, whose entries are independent up to symmetry, and whose variances are the same as for a random tensor of \hyperlink{model:sc}{Model $S_\C$}.\\
    One finds
    \begin{equation*}
        \E_{G}\left(\lvert\langle G,u^{\otimes p}\rangle \rvert^{2k} \right)= \frac{k!}{d^k}.
    \end{equation*}
    
    To prove the upper bound when $G$ is replaced by $B$ we adopt the same strategy as in the proof of Lemma \ref{lem:rigid-subgaussian-moment-bound}, that is we expand the Hermitian product as a sum over components of $B$, factor the $B$ average of products of components of $B$ over the independent ones, upper bound those factors by the corresponding averages with $G$ components replacing those of $B$, and then repackage the sum to obtain $\E_{B}\left(\lvert\langle B,u^{\otimes p}\rangle \rvert^{2k} \right)\le \E_{G}\left(\lvert\langle G,u^{\otimes p}\rangle \rvert^{2k} \right).$

   The proof of inequality \eqref{eq:symmetric-tilde:rigid-sub-Gaussian-bound} is a consequence of Lemma \ref{lem:rigid-subgaussian-moment-bound}. Simply notice that, for all unit norm complex vector $u\in \C^d$, one has
    \begin{equation*}
        \langle \widetilde{B},u^{\otimes p}\rangle =\langle P_{\sym,p}^{(d)}(T),u^{\otimes p}\rangle = \langle T, P_{\sym,p}^{(d)}(u^{\otimes p})\rangle=\langle T, u^{\otimes p}\rangle,
    \end{equation*}
    where we use the hermiticity of $P_{\sym,p}^{(d)}$ and $u^{\otimes p}\in \Sym_p(\C^d).$ Since $T$ is from \hyperlink{model:ak}{Model $A_\C$}, it follows from Lemma \ref{lem:rigid-subgaussian-moment-bound} for $u_1=\ldots=u_p=u$ and $d_1=\ldots=d_p=d$, that
    \begin{equation*}
        \E_{\widetilde{B}}\left( \lvert\langle \widetilde{B},u^{\otimes p}\rangle\rvert^{2k}\right)=\E_T\left(\lvert\langle T,u^{\otimes p}\rangle\rvert^{2k} \right)\le \frac{k!}{d^k}.
    \end{equation*}
\end{proof}
\begin{proof}[Proof of Theorem \ref{thm:symm_complex_tensors_bound}]

The proof of this is very similar to the proof of Theorem \ref{thm:asymmetric}.\\ 

\noindent{\it \RomanNum{1}. $d$ fixed and $p \to \infty.$} \\

From equations \eqref{eq:symmetric_inj_norm_bounded} and \eqref{eq:symmetric:rigid-sub-Gaussian-bound} or \eqref{eq:symmetric-tilde:rigid-sub-Gaussian-bound}, we obtain
\[
    \sqrt{d} \E[\injnormC{B}] \leq \sqrt{d} \E[\injnormC{B}^{2k}]^{\frac{1}{2k}} \leq \left( k! \binom{d+pk-1}{pk} \right)^{\frac{1}{2k}}
\]
for any positive integer $k$. Now we take $k = \lfloor \alpha \log p \rfloor$ for $\alpha > 0$ to be determined. Arguing as in \eqref{eqn:p-to-infinity-complex-calculus-first-stirlings}, we find that for any $\epsilon > 0$ and large enough $k$ (equivalently, large enough $p$) we have
\[
    (k!)^{\frac{1}{2k}} \leq \sqrt{\alpha \log p}{e} \left( 1 + \left( \frac{1}{4} + \epsilon \right) \frac{\log k}{k} \right).
\]
If $k = \lfloor \alpha \log p \rfloor$, then $\frac{\log k}{k} \leq \frac{\log(\alpha \log p)}{\alpha \log p - 1}$, and since $\lim_{p \to \infty} \frac{\log(\alpha \log p)}{\alpha \log p - 1} \cdot \frac{\alpha \log p}{\log \log p} = 1$, we find that for any $\alpha, \epsilon > 0$ and $p$ large enough, we have
\[
    (k!)^{\frac{1}{2k}} \leq \sqrt{\frac{\alpha \log p}{e}} \left( 1 + (1+\epsilon) \cdot \frac{1}{4\alpha} \cdot \frac{\log\log p}{\log p} \right).
\]
(This is the main difference from before; the correction is now order $\frac{\log \log p}{\log p}$ instead of $\frac{1}{p}$, so it will contribute to the correction term in the theorem statement.) We replace \eqref{eqn:p-to-infinity-complex-calculus-second-stirlings} with
\[
    \lim_{p \to \infty} \left( \frac{\binom{\beta+pk}{pk}^{\frac{1}{2k}}}{\exp\left(\frac{\beta}{2\alpha}\right)} - 1 \right) \frac{\log p}{\log \log p} = \frac{\beta}{2\alpha}
\]
for $k = \lfloor \alpha \log p \rfloor$, and use these to obtain that, for every $\epsilon > 0$,
\begin{align*}
    \sqrt{\frac{d}{\log p}} \E[\injnormC{B}] \leq &\exp\left(\frac{d-1}{2\alpha} + \frac{\log(\alpha) - 1}{2} \right) \\
    &\quad \times \left( 1 + (1+\epsilon) \cdot \frac{1}{4\alpha} \cdot \frac{\log \log p}{\log p}\right) \left( 1 + (1+\epsilon)\frac{d-1}{2\alpha} \cdot \frac{\log \log p}{\log p} \right).
\end{align*}
The fraction $\alpha \mapsto \frac{d-1}{2\alpha} + \frac{\log(\alpha)-1}{2}$ is minimized at $\alpha = d-1$, where it takes the value $\log(d-1)/2$. The dominant term in the correction is actually the cross term, which then has the form 
\[
    (1+\epsilon)\left( \frac{1}{4\alpha} + \frac{2(d-1)}{4\alpha}\right) \frac{\log \log p}{\log p} = (1+\epsilon) \frac{2d-1}{4d-4} \cdot \frac{\log \log p}{\log p}.
\]

\noindent{\it \RomanNum{2}. $p$ fixed and $d \to \infty.$} \\

   Again, from Theorem \ref{thm:deterministic_bound}, equation \eqref{eq:symmetric_inj_norm_bounded} and Lemma \ref{lem:sym-rigid-subgaussian-moment-bound}, equation \eqref{eq:symmetric:rigid-sub-Gaussian-bound} and \eqref{eq:symmetric-tilde:rigid-sub-Gaussian-bound}, we obtain
    \begin{equation*}
        \E_B\injnormC{B}\le\left(\binom{d+pk-1}{pk} \frac{k!}{d^k}\right)^{1/2k}
    \end{equation*}
    and after setting $d=\lfloor\alpha k\rfloor$ one checks
    \begin{equation}
        \lim_{k\to\infty}\left(\binom{\lfloor\alpha k\rfloor+pk-1}{pk} \frac{k!}{\lfloor\alpha k\rfloor^k}\right)^{1/2k}=\phi_p(\alpha),
    \end{equation}
    which after minimizing over $\alpha$ leads to the first claim. We now come to the asymptotic behavior as $p\to \infty$. To this aim, we follow the same route as in the proof of Theorem \ref{thm:asymmetric}. We differentiate $\phi_p(\alpha)$ to find that the derivative vanishes if and only if the function $h_p(\alpha):=\alpha\log\left(1+\frac{p}{\alpha}\right)-1$ vanishes. $h_p$ can be shown to admit a unique positive zero on the positive line. This zero is the unique minimum of $\phi_p$. We upper bound the value of $\phi_p$ using a similar technique. We set $y=p/\alpha$ so that $h_p(\alpha)=0 \Leftrightarrow p(\log y +\log\left(1+1/y \right))=y$. We consider the equation $p\log y'=y'$ on an auxiliary variable $y'$ instead. Using the Lambert function, this latter equation is solved by $y'=-p\mathcal{W}_{-1}(-1/p)$. We therefore obtain $\alpha'=p/y'=\frac1{\log p}\left(1-\frac{\log \log p}{\log p}+o\left(\frac{\log \log p }{\log p}\right)\right)$ which suffices to prove the last claim.
\end{proof}

\begin{remark}
We remark that, combined with prior results of Friedland and Kemp \cite{friedland2018most}, our result shows that rigidly sub-Gaussian bosonic states are almost as entangled as possible. Indeed, fix $d \geq 2$ and let $p \to \infty$. Let $B$ come from \hyperlink{model:sc}{Model $S_\C$}, and define
\[
    \ket{\psi_B} \defeq \frac{B}{\|B\|_{\textup{HS}}}.
\]
Corollary \ref{cor:comparing-friedland-kemp} below shows that $\limsup_{p \to \infty} \E[\injnormC{\ket{\psi_B}}] \leq \sqrt{\frac{(d-1)!(d-1)(\log p)}{p^{d-1}}}$; on the other hand, Theorem 1.1 of \cite{friedland2018most} shows that, deterministically for each $d$ and $p$, we have $\injnormC{\ket{\psi_B}} \geq \left( \binom{d+p-1}{p} \right)^{-1/2}$. As $p \to \infty$, this scales like $\sqrt{\frac{(d-1)!}{p^{d-1}}} (1 + \oo_{p \to \infty}(1))$. Thus, our results show that $\injnormC{\ket{\psi_B}}$ is within a factor $\sqrt{(d-1)(\log p)}$ of being as small as possible, and thus (since the entanglement is twice the negative logarithm of the injective norm) indeed near-maximally entangled.
\end{remark}

\begin{remark}
In the physics literature, a common folklore claim is that symmetric tensors are less entangled than (i.e., have larger injective norm than) their asymmetric counterparts, when both are normalized to have unit Frobenius norm. (For comments on this intuition see \textit{e.g.} \cite{friedland2018most,nakata2020generic} and \cite[Figure 7]{FitLanNec2025} for numerics compatible with this observation in the $p=3$ case.) In this remark, we explain how our results give a formalization of this heuristic.

Indeed, let $T$ be a tensor from Model \hyperlink{model:ak}{$A_\C$}, with underlying $\C$-rigidly-sub-Gaussian distribution $\gamma$, and let $B$ and $\widetilde{B}$ be tensors from Models \hyperlink{model:sc}{$S_\C$} and \hyperlink{model:stildec}{$\widetilde{S}_\C$}, respectively, with the same underlying $\gamma$. Let $\ket{\psi_T} = T/\|T\|_{\textup{HS}}$, $\ket{\psi_B} = B/\|B\|_{\textup{HS}}$, and $\ket{\psi_{\widetilde{B}}} = \widetilde{B}/\|\widetilde{B}\|_{\textup{HS}}$. Take for concreteness the case $d$ fixed and $p \to \infty$. We claim that, for every $\epsilon > 0$ and $p$ large enough, we have
\begin{equation}
\label{eqn:symmetric-vs-asymmetric}
    \E[\injnormC{\ket{\psi_T}}] \leq (1+\epsilon) \sqrt{\frac{(d-1)p\log p}{d^p}} \ll \frac{1}{\sqrt{\binom{d+p-1}{p}}} \leq \min(\injnormC{\ket{\psi_B}},\injnormC{\ket{\psi_{\widetilde{B}}}}).
\end{equation}
In \eqref{eqn:symmetric-vs-asymmetric}, the third inequality is true \emph{deterministically}, and follows from \cite{friedland2018most} as mentioned above; the second inequality $\ll$ is just an emphatic version of $\leq$, emphasizing that the left-hand side decays exponentially in $p$ and the right-hand side decays only polynomially in $p$; and the first inequality is what we prove now. If $G$ is the event that $\|T\|_{\textup{HS}} \leq (1-\delta) d^{(p-1)/2}$, then since $\injnormC{\ket{\psi_T}} = \injnormC{T}/\|T\|_{\textup{HS}}$, we have
\begin{align*}
    \E[\injnormC{\ket{\psi_T}}] &\leq \frac{1}{(1-\delta) d^{(p-1)/2}} \E[\injnormC{T}] + \E[\injnormC{\ket{\psi_T}}\mathds{1}_G] \\
    &\leq \frac{1+\epsilon}{1-\delta} \sqrt{\frac{(d-1)p\log p}{d^p}} + \P(G).
\end{align*}
In the last inequality, we used Theorem \ref{thm:asymmetric} and the deterministic inequality $\injnormC{\ket{\psi_T}} \leq 1$. Since $\P(G) = \P\Bigl(\frac{d\|T\|_{\textup{HS}}^2}{d^p} \leq (1-\delta)^2\Bigr)$, and $d\|T\|_{\textup{HS}}^2$ is the sum of $d^p$ independent variables of mean one and finite exponential moment, Chernoff bounds show that, for each $\delta > 0$, there exists $I(\delta)$ such that $\P(G) \leq \exp(-d^pI(\delta))$. Since this tends to zero much more quickly than the first term, we can absorb it at the cost of redefining $\epsilon$ and $\delta$.

In the Gaussian case, the first inequality in \eqref{eqn:symmetric-vs-asymmetric} also follows from our prior work \cite{dartois2024injective}, but the non-Gaussian cases are new to the best of our knowledge.
\end{remark}


\section{Application: Bounded-rank tensors} 
\label{sec:bounded_rank}

In this section, we consider the random tensors of \hyperlink{model:bc}{Model $B_\C$}. 
\begin{remark}
Such random tensors model random quantum states with limited multipartite Schmidt rank entanglement \cite{eisert2001schmidt,severini2006two}. The bound on the multipartite Schmidt rank of those states implies that the \emph{bipartite} Schmidt rank about any bipartition is also bounded above by $R$ (see \textit{e.g.} \cite{landsberg2011tensors} through the rank of flattenings of a tensor). Therefore, those states have a matrix product state representation with local bond dimensions $D_i\le R$ for all $1\le i\le p$ (this is a consequence of the algorithm described in \cite[section 4.1.3]{schollwock2011density}). They are easier to produce algorithmically (and possibly experimentally on a quantum computer) as they can be realized by low-depth quantum circuits \cite{Malz2024preparation}. It can be noted that those states are used in approximation methods of computational quantum chemistry \cite[section 2.4]{beck2000multiconfiguration} to extend simpler approximation schemes relying on product states. 
\end{remark}

Per the above remark, from a quantum information standpoint, it is natural to consider random complex tensors of bounded rank as they form a relevant family in several problems. However, one can imagine problems where a similar model of real random tensors of bounded rank is relevant. We expect that there is no obstruction in extending the content of this section to such a real model. However, for the sake of brevity, we prefer to focus solely on the complex version here. 
\\

Commenting on the rank of realizations of tensors of \hyperlink{model:bc}{Model $B_\C$}, we first note that for $R$ smaller than $d$, one has the property
\begin{proposition}\label{prop:rank_property}
    Let $p>2$,  $R<d$, and let the law of the entries of the random vectors $x^{(i)}_k$ be absolutely continuous with respect to Lebesgue measure. Then $\text{rank } T = R$ almost surely.
\end{proposition}

We prove this result in Appendix \ref{app:rank_property}. The relevant definitions, including the notion of tensor rank, are recalled there. Our main result on this model is given below. 

\begin{theorem}\label{thm:bounded_rank}
    Let $T$ be distributed according to the bounded-rank random tensor \hyperlink{model:bc}{Model $B_\C$} of rank $R$. Then for all fixed $d,k\in \mathbb{N}^*,$ 
    \begin{equation}\label{eq:finite_d_k_bounded_rank}
        \E[\injnormC{T}]\le\frac1{d^{p/2}}\left( (k!)^2\sum_{\mathbf{a}\in\N^R, \sum_s\mathbf{a}_s=k} \prod_{s=1}^R (\mathbf{a}_s!)^{p-2}\right)^{\frac1{2k}}.
    \end{equation}
    Moreover,
    \begin{equation}\label{eq:asymptotic_bound_bounded_rank}
        \lim_{d\to \infty}\E[\injnormC{T}]= 1.
    \end{equation}
\end{theorem}

\begin{remark}
    In the simple setting $R=1$, then $T=x^{(1)}_1\otimes \ldots\otimes x_p^{(1)}$ is a random tensor of rank $1$. Then for each realization of $T$ the injective norm is $$\injnorm{T}=\prod_{k=1}^p\left\langle x_k^{(1)},\frac{x_k^{(1)}}{\lnorm x_k^{(1)}\rnorm_2}\right\rangle=\prod_{k=1}^p\lnorm x_k^{(1)}\rnorm,$$ that is the product of norms of $p$ independent complex Gaussian vectors. It is a simple exercise to prove that $\lim_{d\to\infty}\E_T\injnorm{T}=1.$
\end{remark}

\begin{remark}\label{rk:comparison}
    It is tempting to compare the bound of Theorem \ref{thm:bounded_rank} to the bound of equation \eqref{eq:multivariate_bound} for $\eta_i=1$ in Theorem \ref{thm:asymmetric}. However, doing so requires putting the variance of the entries of the tensors on the same footing. Rescaling the entries so that the variance is set to $1$ in both cases indicates that the injective norm lives at the scale $\sqrt{dp\log p}$ in the Gaussian case, while in the bounded-rank case the scale is $\sqrt{d^p/R}.$
\end{remark}

\begin{remark}
    Models of this flavor have appeared several times before in the literature. Most prior results are for real symmetric analogues of our model, i.e. for $\sum_{j=1}^R (x^{(j)})^{\otimes p}$, where $\{x^{(j)}\}$ is a collection of random vectors in $\R^d$. These models have appeared under a variety of names, such as ``Wishart tensor''\footnote{\cite{bonnin2024universality} uses the name Wishart tensor for a different type of random tensors, and there is no standard use for this name yet. The case of our model is also a bit different because of the lack of symmetry, which forbids the $p=2$ case to become the Wishart random matrix case. } -- so-called since the matrix case $p = 2$ reduces to the classical Wishart \cite{wishart1928generalised} random matrix $W = GG^T$, where $G$ is the matrix whose columns are the random vectors $x^{(j)}$ -- and ``simple tensor.'' Papers dealing with such models include \cite{dhoyer2024limit,Zhi2024,AbdVer2026,AlGCheSan2025,CheSan2025}.
    
    Among these, the closest to \hyperlink{model:bc}{Model $B_\C$} is \cite{CheSan2025}. A special case of their Theorem 2.1 gives the following result: Let $R \in \N$ and let $x^{(i)}_1,\ldots,x^{(i)}_p$ be, for each $1 \leq i \leq R$, a family of $p$ independent vectors with i.i.d. centered sub-Gaussian entries with variance set to $1/d$, independent for different $i$'s. Consider the random tensor $T = \sum_{i=1}^R x^{(i)}_1 \otimes \cdots \otimes x^{(i)}_p$. Then there exists a sequence $(C_p)$ such that
    \[
        \E[\injnormR{T}] \leq C_p \frac{R}{d^{p/2}} \left[ \left( \frac{dp}{R}\right)^{1/2} + \frac{1}{R}(d+\log(R))^{p/2}\right],
    \]
    and if the entries are actually Gaussian, then there is an analogous lower bound with $C_p$ replaced by another sequence $(c_p)$. If $p$ and $R$ are fixed and $d \to \infty$, these bounds behave asymptotically like $C_p$ and $c_p$, respectively. Although the results are not directly comparable because they deal with different models, Theorem \ref{thm:bounded_rank} amounts to the statement that one can take $c_p = C_p = 1$ asymptotically in our analogue. However, \cite[Theorem 2.1]{CheSan2025} deals with the more general case where the $x^{(i)}_1$ variables live in some Hilbert space and have a covariance operator which is not necessarily identity, and allows for sub-Gaussian entries which are not necessarily rigidly sub-Gaussian.
\end{remark}

The proof relies on lemmas that we spell out in the following paragraphs. First we introduce a notation. For all $\mathbf{i}=(i_1,\ldots,i_k)\in [R]^k$ a multi-index, we denote $q(\mathbf{i})=(q_1(\mathbf{i}),\ldots,q_R(\mathbf{i}))$ the multiplicity vector of $\mathbf{i}$ whose entries are $q_j(\mathbf{i})=\#\{i_s:i_s=j\}$. For all $\mathbf{i}=(i_1,\ldots,i_k)\in [R]^k$ we denote $\{\mathbf{i}\}$ or $\{i_1,\ldots,i_k\}$ the corresponding multi-set. To avoid introducing too many notations, we also denote the associated multiplicity vector by $q(\mathbf{i})$.

\begin{lemma}\label{lem:rank1-average}
    Let $k\in \mathbb{Z}_+$ and $\mathbf{i}=(i_1,\ldots, i_k),\mathbf{j}=(j_1,\ldots, j_k)\in[R]^k$ be multi-indices. Let the family $\{x_c^{(i)}\}_{i\in [R],c\in [p]}$ be random vectors as in \hyperlink{model:bc}{Model $B_\C$}. One has, for all $u_c\in \mathbb{S}_\C^{d-1}$ unit norm complex vector,
    \begin{equation}
        \langle u_c^{\otimes k}\rvert\E_{\{x_c^{(i)}\}}\left(\lvert x_c^{(i_1)}\otimes \ldots\otimes x_c^{(i_k)}\rangle\langle x_c^{(j_1)}\otimes \ldots\otimes x_c^{(j_k)}\rvert\right)\lvert u_c^{\otimes k}\rangle\le\delta_{\{i_1,\ldots,i_k\},\{j_1,\ldots,j_k\}}\prod_{s=1}^Rq_s(\mathbf{i})!.
    \end{equation}
\end{lemma}

\begin{proof}
Consider first a family of Gaussian vectors $\{y_c^{(i)}\}$ with i.i.d. circularly symmetric entries of variance $1/d$. Then, the proof that
$$\langle u_c^{\otimes k}\rvert\E_{\{y_c^{(i)}\}}\left(\lvert y_c^{(i_1)}\otimes \ldots\otimes y_c^{(i_k)}\rangle\langle y_c^{(j_1)}\otimes \ldots\otimes y_c^{(j_k)}\rvert\right)\lvert u_c^{\otimes k}\rangle=\delta_{\{i_1,\ldots,i_k\},\{j_1,\ldots,j_k\}}\prod_{s=1}^Rq_s(\mathbf{i})!$$
is a calculation. The key point to consider is that $y_c^{(i)},y_c^{(j)}$ are independent for $i\neq j$, so that
\begin{enumerate}[label=\alph*. , ref=\alph*.]
  \item the multi-indices $\mathbf{i},\mathbf{j}$ must be the same when seen as multisets $\{\mathbf{i}\}=\{\mathbf{j}\}$ for the expectation to be non zero.
  \label{itm:multisets-eq}
  \item the expectation factorizes over group of indices $i_s$ having the same value.
  \label{itm:factorizing-expectation}
\end{enumerate}
The rest is a consequence of properties of complex Gaussian vectors together with the fact that $\langle u_c,u_c\rangle=1$.\\
We are left to prove that 
\begin{multline*}
\langle u_c^{\otimes k}\rvert\E_{\{x_c^{(i)}\}}\left(\lvert x_c^{(i_1)}\otimes \ldots\otimes x_c^{(i_k)}\rangle\langle x_c^{(j_1)}\otimes \ldots\otimes x_c^{(j_k)}\rvert\right)\lvert u_c^{\otimes k}\rangle\\
\le \langle u_c^{\otimes k}\rvert\E_{\{y_c^{(i)}\}}\left(\lvert y_c^{(i_1)}\otimes \ldots\otimes y_c^{(i_k)}\rangle\langle y_c^{(j_1)}\otimes \ldots\otimes y_c^{(j_k)}\rvert\right)\lvert u_c^{\otimes k}\rangle.
\end{multline*}
This is done in the same way as in the proof of Lemma \ref{lem:rigid-subgaussian-moment-bound}. We compare rigid sub-Gaussian to Gaussian. In particular, because of point \ref{itm:factorizing-expectation} above, it is sufficient to prove that for all $q\in \mathbb{N}$, $u$ a uniform unit sphere vector, $x$ a vector with circularly symmetric, $\C$-rigidly sub-Gaussian i.i.d. entries of variance $1/d$, and $y$ a Gaussian vector with i.i.d. circularly symmetric entries of variance $1/d$, one has
$$
\langle u^{\otimes q}\lvert\E_{x}\left( \vert x^{\otimes q}\rangle\langle x^{\otimes q}\vert\right) \rvert u^{\otimes q}\rangle\le \langle u^{\otimes q}\lvert\E_{y}\left( \vert y^{\otimes q}\rangle\langle y^{\otimes q}\vert\right) \rvert u^{\otimes q}\rangle.
$$
Expanding the left hand side, one reaches
\begin{equation*}
    \langle u^{\otimes q}\lvert\E_{x}\left( \vert x^{\otimes q}\rangle\langle x^{\otimes q}\vert\right)\rvert u^{\otimes q}\rangle=\sum_{\substack{i_1,\ldots, i_p\in [d]\\
    j_1,\ldots, j_p\in [d]}}(u^{\otimes q})_{i_1,\ldots,i_p}(\bar{u}^{\otimes q})_{j_1,\ldots, j_p}\E_{x}\left(\prod_{\eta=1}^q x_{i_\eta}\bar{x}_{j_\eta}\right).
\end{equation*}
We can factorize the expectation $\E_{x}$ over independent components of $x$ by grouping those together. Formally, we write
\begin{align*}
    &\sum_{\substack{i_1,\ldots, i_p\in [d]\\
    j_1,\ldots, j_p\in [d]}}(u^{\otimes q})_{i_1,\ldots,i_p}(\bar{u}^{\otimes q})_{j_1,\ldots, j_p}\E_{x}\left(\prod_{\eta=1}^q x_{i_\eta}\bar{x}_{j_\eta}\right) \\
    &= \sum_{\pi\in \mathcal{P}_{\text bal}([q]\sqcup[\bar q])}\sum_{\substack{i_b\in [d]:b\in \pi\\i_b\neq i_{b'},\forall b\neq b'}}\prod_{b\in \pi}\lvert u_{i_b}\rvert^{\lvert b \rvert}\prod_{b\in \pi}\E_{x}\left( \lvert x_{i_b}\rvert^{\lvert b \rvert}\right),
\end{align*}
using the rigid sub-Gaussian property that $\E_{x}\left( \lvert x_{i_b}\rvert^{\lvert b \rvert}\right)\le \E_{y}\left( \lvert y_{i_b}\rvert^{\lvert b \rvert}\right)$ and repackaging the right-hand-side sum above with the replacement of $\E_{x}\left( \lvert x_{i_b}\rvert^{\lvert b \rvert}\right)$ by $\E_{y}\left( \lvert y_{i_b}\rvert^{\lvert b \rvert}\right)$, we obtain
$$\langle u^{\otimes q}\lvert\E_{x}\left( \vert x^{\otimes q}\rangle\langle x^{\otimes q}\vert\right)\rvert u^{\otimes q}\rangle\le\langle u^{\otimes q}\lvert\E_{y}\left( \vert y^{\otimes q}\rangle\langle y^{\otimes q}\vert\right) \rvert u^{\otimes q}\rangle, $$
as claimed and which suffices to prove the lemma.
\end{proof}
\begin{lemma}\label{lem:double_bound_constrained_sum}
 Let $p>2$ and $k\ge R+1$ be integers. Then one has
 \begin{equation}
     \left(\sum_{\mathbf{a}\in\N^R, \sum_s\mathbf{a}_s=k} \prod_{s=1}^R (\mathbf{a}_s!)^{p-2}\right)^{\frac1{2k}}\le e^{\frac12(p-2)(\log k-1)} \mathfrak{P}(k)
 \end{equation}
 where $\lim_{k\to \infty}\mathfrak{P}(k)=1.$
\end{lemma}
\begin{proof}
Start by splitting the sum according to the number of nonzero $\mathbf{a}_s$ values into 
\begin{equation*}
    \sum_{\mathbf{a}\in\N^R, \sum_s\mathbf{a}_s=k} \prod_{s=1}^R (\mathbf{a}_s!)^{p-2}=\sum_{m=1}^R\binom{R}{m}C^{(p)}_{m,k}, \quad C^{(p)}_{m,k}:=\sum_{\substack{\mathbf{a}\in (\N^*)^m, \sum_s\mathbf{a}_s}=k}\prod_{s=1}^m(\mathbf{a}_s!)^{p-2}.
\end{equation*}
We have using \cite[equation (76) \& equation (112)]{Gronwall1918-Gamma} 
\begin{equation*}
       \Gamma(z)<\sqrt{\frac1{2\pi z}}\left( \frac{z}{e}\right)^ze^{\frac1{12z}},
\end{equation*}
so that
\begin{align*}
    C^{(p)}_{m,k}&\le \sum_{\substack{\mathbf{a}\in (\N^*)^m, \sum_s\mathbf{a}_s}=k}\prod_{s=1}^m\left[\left(\sqrt{\frac1{2\pi (\mathbf{a}_s+1)}}\right)^{p-2}e^{(p-2)[(\mathbf{a}_s+1)\log(\mathbf{a}_s+1)-\mathbf{a}_s-1]+\frac{(p-2)}{12(\mathbf{a}_s+1)}}\right].
\end{align*}
For each value of $m\in \{1,\ldots,R\}$, change variables to $\mathbf{a}=k x,$ with $x\in \Delta_m^{\text{int},k}\subset \Delta_m$ a probability vector living in the subset  $\Delta_m^{\text{int},k}$ of the standard simplex $\Delta_m$, defined as 
$$\Delta_m^{\text{int},k}:=\{x\in \Delta_m: \forall s \in [m], \exists i_s\in \mathbb{N}^*, x_s=i_s/k\}.$$ Hence,
\begin{align*}
C^{(p)}_{m,k}&\le \sum_{x\in \Delta_m^{\text{int},k}}\prod_{s=1}^m\left[\left(\sqrt{\frac1{2\pi (k x_s+1)}}\right)^{p-2}e^{(p-2)[(kx_s+1)\log(kx_s+1)-kx_s-1]+\frac{(p-2)}{12(kx_s+1)}}\right]\\
&=(2\pi)^{-\frac{m}{2}(p-2)}\sum_{x\in \Delta_m^{\text{int},k}}e^{(p-2)k(\log(k)-1)}e^{(p-2)\frac{m}{2}\log(k)}e^{(p-2)\sum_{s=1}^m\left[\left(kx_s+\frac12\right)\log(x_s+\frac1k)-1+\frac1{12(kx_s+1)}\right]}.
\end{align*}
Note then that
\begin{align*}
    &\sum_{s=1}^m\left[(kx_s+1/2)\log\left(x_s+\frac1k\right)-1+\frac1{12(kx_s+1)}\right] \\
    &\le k\sum_{s=1}^m\left(x_s+\frac1{2k}\right)\log\left(x_s+\frac1{k} \right)\le 2k\log\left(1+\frac1{k}\right)
\end{align*}
which holds since $m \leq R \leq k$. Hence, by a simple count of probability vectors in $\Delta_m^{\text{int},k}$,
\begin{align*}
C^{(p)}_{m,k}&\le(2\pi)^{-\frac{m}{2}(p-2)}e^{(p-2)k\left[\log(k)-1+2\log(1+1/k)\right]}e^{\frac{m}{2}(p-2)\log(k)}\binom{k-1}{m-1}.
\end{align*}
We get 
\begin{align*}
    \sum_{m=1}^R\binom{R}{m}C^{(p)}_{m,k}&\le \sum_{m=1}^R(2\pi)^{-\frac{m}{2}(p-2)}e^{(p-2)k\left[\log(k)-1+2\log(1+1/k)\right]}e^{\frac{m}{2}(p-2)\log(k)}\binom{R}{m}\binom{k-1}{m-1}\\
    &\le(2\pi)^{-\frac{1}{2}(p-2)}e^{(p-2)k\left[\log(k)-1+2\log(1+1/k)\right]}e^{\frac{R}{2}(p-2)\log(k)} R!\frac{(k-1)^{R+1}-(k-1)}{k-2}
\end{align*}
where we coarsely bound: 
$$(2\pi)^{-\frac{m}{2}(p-2)}\le (2\pi)^{-\frac{1}{2}(p-2)},\ e^{\frac{m}{2}(p-2)\log(k)}\le e^{\frac{R}{2}(p-2)\log(k)},\ \binom{R}{m}\le R!,\ \binom{k-1}{m-1}\le (k-1)^m.$$
Hence, defining 
\[
    \mathfrak{P}(k):=\left(R!(2\pi)^{-\frac{1}{2}(p-2)}e^{2(p-2)k\log(1+1/k)}e^{\frac{R}{2}(p-2)\log(k)}\frac{(k-1)^{R+1}-(k-1)}{k-2}\right)^{1/2k},
\]
one has
$$\left(\sum_{\mathbf{a}\in\N^R, \sum_s\mathbf{a}_s=k} \prod_{s=1}^R (\mathbf{a}_s!)^{p-2}\right)^{1/2k}\le e^{\frac{1}{2}(p-2)(\log(k)-1)}\mathfrak{P}(k).$$
It is a simple exercise to check that $\lim_{k\to\infty}\mathfrak{P}(k)=1,$ proving the claim.
\end{proof}

We are now ready to come to the proof of Theorem \ref{thm:bounded_rank}.
\begin{proof}[Proof of Theorem \ref{thm:bounded_rank}.]
We use the upper bound on the injective norm of Theorem \ref{thm:deterministic_bound}, equation \eqref{eq:inj_norm_bounded}, which we average over $T$ for $T$ a random tensor of the bounded-rank \hyperlink{model:bc}{Model $B_\C$}. We therefore have to compute
\begin{multline*}
\E_T\E_{u_i}\left( \lvert \langle T,u_1\otimes \ldots\otimes u_p \rangle\rvert^{2k} \right)=\\
\sum_{\substack{\mathbf{i}=(i_1,\ldots, i_k)\in[R]^k\\ \mathbf{j}=(j_1,\ldots, j_k)\in[R]^k}}\prod_{c=1}^p\E_{u_c}\E_T(\langle u_c^{\otimes k},x_c^{(i_1)}\otimes \ldots\otimes x_c^{(i_k)}\rangle\langle x_c^{(j_1)}\otimes \ldots\otimes x_c^{(j_k)},u_c^{\otimes k}\rangle )
\end{multline*}
which we reduce by Lemma \ref{lem:rank1-average} to 
\begin{multline*}
    \E_T\E_{u_i}\left( \lvert \langle T,u_1\otimes \ldots\otimes u_p \rangle\rvert^{2k} \right)\le\frac1{d^{pk}}\sum_{\substack{\mathbf{i}=(i_1,\ldots, i_k)\in[R]^k\\ \mathbf{j}=(j_1,\ldots, j_k)\in[R]^k}}\delta_{\{i_1,\ldots,i_k\},\{j_1,\ldots,j_k\}}\prod_{s=1}^R(q_s(\mathbf{i})!)^p\\
    =\frac1{d^{pk}}\sum_{\substack{\mathbf{i}=(i_1,\ldots, i_k)\in[R]^k}} \frac{k!}{\prod_{s=1}^Rq_s(\mathbf{i})!}\prod_{s=1}^R(q_s(\mathbf{i})!)^p
    =\frac{(k!)^2}{d^{pk}}\sum_{\mathbf{a}\in\N^R, \sum_s\mathbf{a}_s=k} \prod_{s=1}^R (\mathbf{a}_s!)^{p-2}.
\end{multline*}
Taking the $(2k)$\textsuperscript{th} root and Jensen's inequality proves the first claim. Additionally, by Theorem \ref{thm:deterministic_bound} and Lemma \ref{lem:double_bound_constrained_sum}, we have for all $k\ge R+1$,
\begin{equation*}
  \E_T\injnormC{T}\le \left[(k!^2)\binom{d+k-1}{k}^p\right]^{\frac1{2k}}\frac{1}{d^{p/2}}e^{\frac12(p-2)(\log k-1)}\mathfrak{P}(k).  
\end{equation*}
We now set $d=\lfloor\alpha k\rfloor$ for $\alpha>0$ and compute the limit of the left hand side as $k\to \infty.$ Minimizing the result over $\alpha$, this leads to 
\begin{equation}
    \lim_{d\to \infty}\E_T\injnorm{T}\le \inf_{\alpha\in \R_+} \frac1{e^{p/2}}\left( \alpha^{-\frac12 (\alpha+1)}(1+\alpha)^{\frac12(\alpha+1)}\right)^p.
\end{equation}
The infimum over $\alpha$ is easily shown to be $1$, achieved as $\alpha \to +\infty$, finishing the proof of the upper bound.\\

For a matching lower bound, we first notice that, for any $T=\sum_{i=1}^Rx_1^{(i)}\otimes \ldots\otimes x_p^{(i)}$, we can lower-bound the injective norm by testing against the normalized version of the first summand in $T$:
\begin{align*}
    \E[\injnormC{T}] &\geq \E\left[ \left\lvert \left\langle T,\frac{x_1^{(1)}}{\lnorm x_1^{(1)} \rnorm_2}\otimes \ldots\otimes \frac{x_p^{(1)}}{\lnorm x_p^{(1)} \rnorm_2}\right\rangle\right\rvert\right] \geq \abs{ \E\left[\left\langle T,\frac{x_1^{(1)}}{\lnorm x_1^{(1)} \rnorm_2}\otimes \ldots\otimes \frac{x_p^{(1)}}{\lnorm x_p^{(1)} \rnorm_2}\right\rangle\right] }. 
\end{align*}
Since we assumed that rigidly sub-Gaussian variables have symmetric distributions, the random variables $\langle x^{(i)}_k,x^{(1)}_k \rangle/\|x^{(1)}_k\|_2$ are centered for $i > 1$ and independent for different $k$ values; thus
\[
    \E\left[\left\langle T,\frac{x_1^{(1)}}{\lnorm x_1^{(1)} \rnorm_2}\otimes \ldots\otimes \frac{x_p^{(1)}}{\lnorm x_p^{(1)} \rnorm_2}\right\rangle\right] = \E\left[\prod_{k=1}^p \|x^{(1)}_k\|_2 + \sum_{i=2}^R \prod_{k=1}^p \ip{x^{(i)}_k, \frac{x^{(1)}_k}{\|x^{(1)}_k\|_2}} \right] = \E[\|x^{(1)}_1\|_2]^p.
\]
If $x$ is a random vector of length $d$ with i.i.d. complex entries of mean zero and variance $1/d$, for any $\epsilon > 0$ we have
\[
    \E[\|x\|_2] \geq \E[\|x\|_2\mathds{1}_{\|x\|_2^2 \geq (1-\epsilon)^2}] \geq (1-\epsilon)\P(\|x\|_2^2 \geq (1-\epsilon)^2).
\]
From the weak law of large numbers, we conclude that the latter probability tends to one as $d \to +\infty$; taking $\epsilon \to 0$, we find $\liminf_{d \to \infty} \E[\|x\|_2] \geq 1$. Thus $\liminf_{d \to \infty} \E[\injnormC{T}] \geq 1$, which finishes the proof.
\end{proof}


\appendix


\section{Numerics}
\label{app:numerics_experiment}
In this appendix we show plots illustrating the results of numerical experiments for the expected injective norm of random tensors of large local dimension $d$. All the numerics were run on a Mac mini M4 Pro, with 24gb of RAM, within a few hours of runtime. On Figure \ref{fig:steinhaus-tensor} we show the case of full tensors with i.i.d. Steinhaus entries, so as to illustrate a case from \hyperlink{model:ak}{Model $A_\K$}.

\begin{figure}[H]
    \centering
    \includegraphics[scale=0.90]{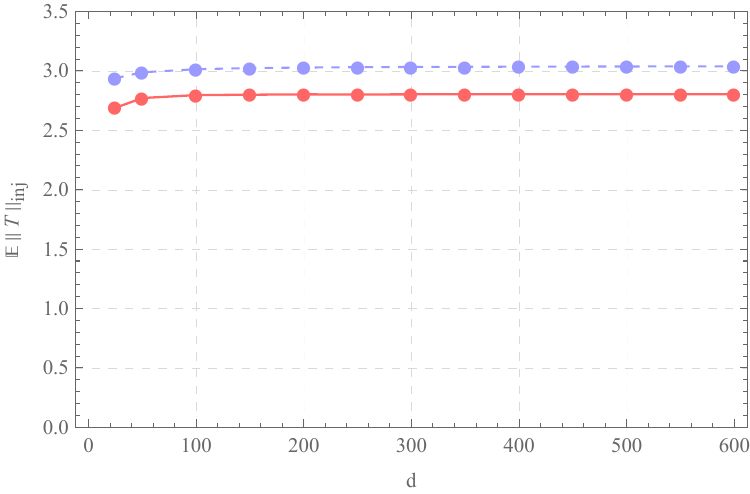}
    \caption{The above plot superimposes: in red and solid lines, the result of computing (a lower bound to) the injective norm by gradient ascent for Steinhaus random tensors for various values of the dimension (the average result over $64$ realizations for each value of the local dimension $d$); and, in blue and dashed lines, the corresponding optimal non-asymptotic upper bound (see Appendix \ref{app:non-asymptotic-upper-bound} and Proposition \ref{prop:finite}).}
    \label{fig:steinhaus-tensor}
\end{figure}
 On Figure \ref{fig:raw-R=3-analytic-finite-asymptotic-vs-numerics} we plot, at rank $R = 3$ for \hyperlink{model:bc}{Model $B_\C$} with Gaussian $x_c^{(i)}$, the numerics obtained via a gradient ascent algorithm and the minimal value over $k$ of the finite $d,k$ bounds of equation \eqref{eq:finite_d_k_bounded_rank} found numerically.  Moreover we also show the limiting $d\to\infty$ bound \eqref{eq:asymptotic_bound_bounded_rank} obtained from Theorem \ref{thm:bounded_rank}. 
\begin{figure}[H]
    \centering
    \includegraphics[scale=0.70]{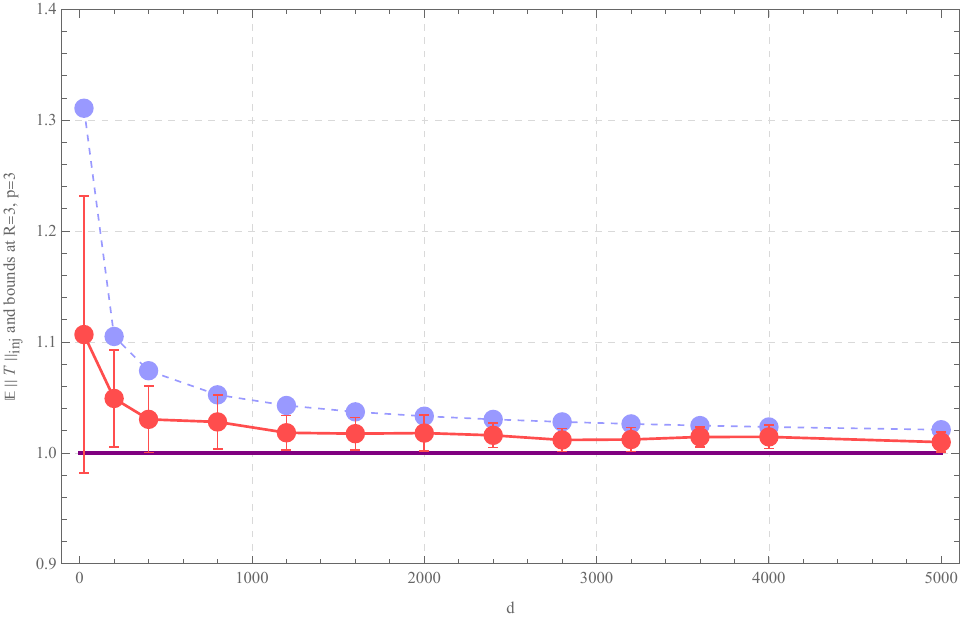}
    \caption{The above plot displays, for each value of $d$: in dashed lines and light blue, the values of the finite non asymptotic upper bounds; in solid lines and light red, the numerics over $40$ realizations and $35$ restarts. It also shows the limiting expectation of \eqref{eq:asymptotic_bound_bounded_rank} as the purple line.  
    }
    \label{fig:raw-R=3-analytic-finite-asymptotic-vs-numerics}
\end{figure}
\begin{figure}[H]
    \includegraphics[scale=0.81]{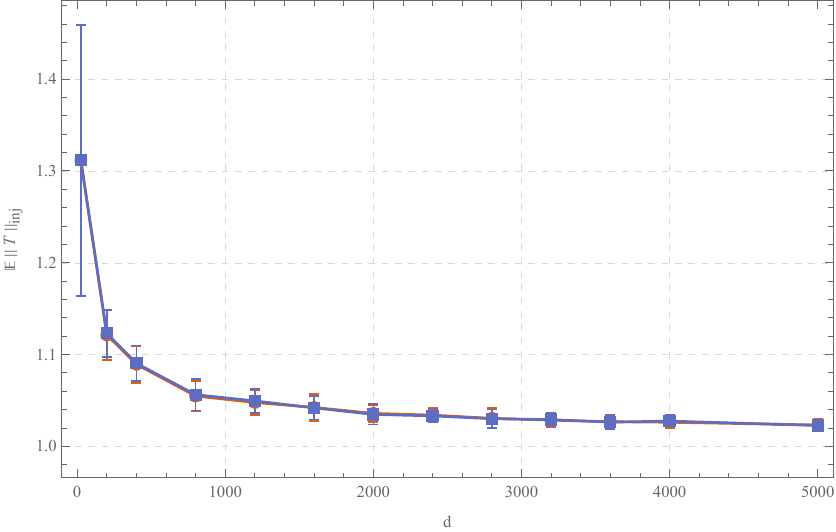}
    \caption{Numerics for the average injective norm of random bounded-rank tensors at rank $R=25$. Two methods are compared: alternating least squares (red) and (noisy projected) gradient ascent method (blue). Bars show the standard deviations. The two methods obtain nearly identical values, so the red data is almost hidden behind the blue data. For each value of the dimension the average of the results is computed over $40$ realizations. Each algorithm uses $35$ restarts to improve the quality of the estimates.\label{fig:pure_num_R25}}
\end{figure}
On Figure \ref{fig:pure_num_R25}, we plot the results of numerics, using two different optimization methods (alternating least squares and gradient method) at the rank $R=25$ for \hyperlink{model:bc}{Model $B_\C$} with Gaussian $x_c^{(i)}$. Together with Figure \ref{fig:raw-R=3-analytic-finite-asymptotic-vs-numerics} this illustrates, given our choice of normalization, the absence of dependence of the limiting value on the rank\footnote{We recall however that the variance carries a rank dependence.} as $d\to \infty$. 


\section{Rank of \texorpdfstring{$T$}{T} in the bounded-rank model}\label{app:rank_property}
\noindent In this appendix we prove Proposition \ref{prop:rank_property}. A good textbook resource for the definitions and mathematics needed is \cite{landsberg2011tensors} with a focus on chapter 5. We recall the definitions we use in our proof below. \\

\noindent The rank is defined as follows (slightly simplified for our context).
\begin{definition}
\label{def:rank}
    If $T\in (\C^d)^{\otimes p}$, then the rank of $T$ is 
    $$\mathrm{rank}\, T:=\min\{r\in \N^*:\exists x_k^{(j)}\in \C^d, 1\le k\le p, 1\le j\le r, T = \sum_{j=1}^rx^{(j)}_1\otimes\ldots\otimes x_p^{(j)}\}.$$
\end{definition}
\noindent The proof uses the classical geometric machinery to define and study the tensor rank, which we recall in part here. Let $N=d^p-1$ and denote by $\P^N$ the complex projective space of dimension $N$. Let $T\in (\C^d)^{\otimes p}$ be a tensor. We denote by $\langle T\rangle$ its class in $\P^N$. 
\begin{definition}[Segre embedding]
    We call the canonical embedding $\sigma:(\P^{d-1})^{\times p}\to \P^N$, the Segre embedding. Let $s_1,\ldots,s_p\in \P^{d-1}$. In projective coordinates $s_i=\langle (s_i)_1:\ldots: (s_i)_p\rangle$ one lets, 
    $$\sigma(s_1,\ldots,s_p)=\langle (s_1)_1(s_2)_1\ldots (s_p)_1: (s_1)_1(s_2)_1\ldots (s_p)_2: \ldots:(s_1)_{d}(s_2)_d\ldots (s_p)_d\rangle.$$
    We denote by $\text{Seg}_p((\P^{d-1})^{\times p}):=\sigma((\P^{d-1})^{\times p}).$
\end{definition}
$\text{Seg}_p((\P^{d-1})^{\times p})$ is the space of separable states, or up to an overall scalar, it is the space of rank $1$ tensors. The spaces of tensors of rank at most $r$ are formalized as the $r$-secant varieties of $\text{Seg}_p((\P^{d-1})^{\times p})$, which we now define.
\begin{definition}[See \cite{landsberg2011tensors}]
  The join $J(X,Y)$ of two varieties $X,Y\subset \P^N$ is the Zariski closure of the set of lines between any point $x\in X, Y\in Y$, that is $J(X,Y)=\overline{\bigcup_{\substack{(x,y)\in X\times Y\\ x\neq y}}\P_{x,y}^1}$. The join of $k$ varieties $X_1,\ldots,X_k\subset \P^N$ is defined by induction $J(X_1,\ldots,X_k):=J(X_1,J(X_2,\ldots,X_k)).$ The $k-$secant variety of a variety $X$ is $\sigma_k(X):=J(\underbrace{X,\ldots,X}_{k \text{ times}}).$
\end{definition}
From \cite[definition 5.2.1.1]{landsberg2011tensors} and comments below it, we extract that $\text{rank }T = \min\{r\in \N^*:\langle T \rangle\in \sigma_r(\text{Seg}_p((\P^{d-1})^{\times p}))\}.$ This gives a geometric characterization of the rank of a tensor that is suited to what we want to show.
\begin{remark}\label{rk:proper_inclusion_secants}
Note that by construction $\sigma_{r-1}(\text{Seg}_p((\P^{d-1})^{\times p}))\subseteq \sigma_{r}(\text{Seg}_p((\P^{d-1})^{\times p}))$. However, by \cite[Proposition 5.3.1.6]{landsberg2011tensors} one has for $R<d$, 
$$\dim \sigma_{R-1}(\text{Seg}_p((\P^{d-1})^{\times p}))<\dim \sigma_R(\text{Seg}_p((\P^{d-1})^{\times p})),$$ 
hence $\sigma_{R-1}(\text{Seg}_p((\P^{d-1})^{\times p}))\subsetneq \sigma_{R}(\text{Seg}_p((\P^{d-1})^{\times p})).$
\end{remark}
\begin{proof}[Proof of Proposition \ref{prop:rank_property}]
We have the necessary ingredients to prove Proposition \ref{prop:rank_property}. Let $T$ be a tensor of the bounded-rank model. Clearly, by  construction $\langle T\rangle\in \sigma_R(\text{Seg}_p((\P^{d-1})^{\times p}))$. We have to show that $\P(\langle T\rangle \in \sigma_{R-1}(\text{Seg}_p((\P^{d-1})^{\times p})))=0.$ To this aim we note that 
$$\Phi: ((\C^d)^{\times p})^{R}\to \C^{d^p}, \quad \{x_k^{(i)}\}_{k,i}\mapsto T=\sum_{i=1}^R x_1^{(i)}\otimes \ldots \otimes x_p^{(i)},$$
is a polynomial map with $\dim \text{Im}(\Phi)=Rp(d-1)+R$. 

From remark \ref{rk:proper_inclusion_secants}, $\{T: \langle T\rangle \in \sigma_{R-1}(\text{Seg}_p((\P^{d-1})^{\times p})) \}$ is an algebraic cone strictly included in $\{T: \langle T\rangle \in \sigma_{R}(\text{Seg}_p((\P^{d-1})^{\times p})) \}$. Hence, the pre-image $\Phi^{-1}(\{T: \langle T\rangle \in \sigma_{R-1}(\text{Seg}_p((\P^{d-1})^{\times p})) \})$ has Lebesgue measure zero in $((\C^d)^{\times p})^{R}$ and so $$\P(\langle T\rangle \in \sigma_{R-1}(\text{Seg}_p((\P^{d-1})^{\times p})))=0.$$
\end{proof}


\section{Application: Numerically efficient non-asymptotic bounds}\label{app:non-asymptotic-upper-bound}

In previous sections, we gave some examples of how to use Theorem \ref{thm:deterministic_bound} to obtain analytic upper bounds for $\E[\injnorm{T}]$ in some limiting sense, for example when $T \in (\R^d)^{\otimes p}$ and either $d \to \infty$ or $p \to \infty$. However, Theorem \ref{thm:deterministic_bound} can also be used to give non-asymptotic bounds on $\E[\injnorm{T}]$ for such a tensor when $d$ and $p$ are both fixed. The purpose of this section is to illustrate how this can be done in a numerically efficient way.

For concreteness, we discuss the case where $T$ comes from \hyperlink{model:ak}{Model $A_\C$}, for simplicity with $d_1 = \cdots = d_p = d$. Combining \eqref{eq:inj_norm_bounded} and Lemma \ref{lem:rigid-subgaussian-moment-bound} we obtain 
\[
    \E[\injnormC{T}] \leq \E[\injnormC{T}^{2k}]^{\frac{1}{2k}} \leq \frac{1}{\sqrt{d}} \left( k! \binom{d+k-1}{k}^p \right)^{\frac{1}{2k}}
\]
for each $k$, and thus one can bound $\E[\injnormC{T}]$ for any fixed $d$ and $p$ simply by taking the infimum of
\[
    M_{d,p}(k) \defeq \frac{1}{\sqrt{d}} \left( k! \binom{d+k-1}{k}^p \right)^{\frac{1}{2k}}
\]
over \emph{discrete} $k \in \{0, 1, 2, \ldots\}$. (The proofs in Section \ref{sec:application_asymmetric} are based on studying asymptotics of this formula, but one can also view it non-asymptotically.) A priori, this could require arbitrarily many evaluations of different $k$ values. However, the following result gives a theoretical upper bound on the required number of evaluations, which depends on $d$ and $p$ in a computationally efficient way.

\begin{proposition}
\label{prop:finite}
For each fixed $d \geq 2$ and $p \geq 3$, $\inf_{k \in \N} M_{d,p}(k)$ is achieved at some $k \in \{1, \ldots, 2pd\log(pd)+1\}$. 
Furthermore, $\inf_{k \in \N} M_{d,p}(k)$ can be found in $O(\log(pd))$ evaluations of the expression $M_{d,p}(k)$.
\end{proposition}

\begin{proof}[Proof of Proposition \ref{prop:finite}]
Notice that, as a function of \emph{continuous} $k$, the quantity $M_{d,p}(k)$ can be rewritten as 
\[
    M_{d,p}(k) = \frac{1}{\sqrt{d}} \exp\left[ \frac{1}{2k} \left( \log(\Gamma(k+1)) + p\log\left(\frac{\Gamma(d+k)}{\Gamma(d)\Gamma(k+1)}\right) \right) \right] =: \frac{1}{\sqrt{d}} \exp\left(\frac{1}{2}g_{d,p}(k)\right).
\]
Suppose we can show that $g_{d,p}$ has a unique minimum over continuous $k > 0$, achieved at some $k_\ast$, for which we can give an explicit upper bound $k_\ast \leq k_u$. Then (a) $\argmin_{k \in \N} g_{d,p}(k) = \argmin_{k \in \N} M_{d,p}(k)$ must be either $k_0 \defeq \lfloor k_\ast \rfloor$ or $k_0 + 1$; (b) $k_u$ is also an upper bound for $k_0$; (c) therefore, $k_0$ can be searched for via a bisection search algorithm, which is shown through the master theorem \cite{cormen2022introduction} to have complexity $O(\log(k_u))$.

Therefore, the result follows from Lemma \ref{lem:finite:h} below, which is essentially an exercise in calculus with special functions, except that (in the notation there) we have to give a simpler upper bound for the Lambert $\mathcal{W}$ term. For this, we use the well-known bound $-\mathcal{W}_{-1}(x) \leq -\log(-x) + \log(-\log(-x)) + 1/2 \leq -2\log(-x)$, say, to obtain 
\[
    k_\ast \defeq x_0 \leq -2(2+x_\ast) \log\left(\frac{e^{\frac{x_\ast-1}{2+x_\ast}}}{2+x_\ast}\right) - 1 \leq 2(2+x_\ast)\log(2+x_\ast) \leq 2pd\log(pd).
\]
\end{proof}

\begin{lemma}
\label{lem:finite:h}
Let $h_{d,p}:\R_+\to \R, \ h_{d,p}(x):=x^2\frac{d}{dx}g_{d,p}(x)$. Let $x_\ast = x_\ast(d,p) = p(d-1)-3/2$. When $d\ge2$ and $p\ge3,$ $h_{d,p}$ has a unique zero at $x_0$, and 
\[
    x_0\le -(2+x_*)\mathcal{W}_{-1}\left(-\frac{e^{\frac{x_*-1}{2+x_*}}}{2+x_*}\right)-1,
\]
where $\mathcal{W}_{-1}$ denotes the real lower branch of the Lambert $\mathcal{W}$ function.
\end{lemma}
\begin{proof}
    First recall that
    \[
    h_{d,p}(x) = x^2\frac{\diff}{\diff x} g_{d,p}(x) = -\log(\Gamma(x+1)) - p\log\left(\frac{\Gamma(d+x)}{\Gamma(d)\Gamma(x+1)}\right) + x(1-p)\psi(x+1)+xp\psi(d+x),
    \]
    where $\psi$ denotes here the digamma function (that is, the logarithmic derivative of the Gamma function). Let also
    \begin{equation}
      \ell_{d,p}(x) \defeq \frac{\diff}{\diff x} h_{d,p}(x) = x(p\psi'(d+x)-(p-1)\psi'(1+x)),  
    \end{equation}
    which can be rewritten, thanks to the property $\psi(x+1)=\psi(x)+\frac1x$, as
    \begin{equation}
        \ell_{d,p}(x) = x\left(\psi'(1+x)-p\sum_{j=1}^{d-1} \frac{1}{(x+j)^2} \right).
    \end{equation}
    Note that $h_{d,p}(0) = h'_{d,p}(0) = 0$, and that for $d \geq 2$
    \[
        h''_{d,p}(0) = \ell'_{d,p}(0) = p\psi'(d)-(p-1)\frac{\pi^2}{6} \leq p\psi'(2) - (p-1)\frac{\pi^2}{6} = \frac{\pi^2-6p}{6}
    \]
    which is negative for $p\ge 2$. Thus $h_{d,p}$ must be negative in a neighborhood of zero; since $\lim_{x \to +\infty} h_{d,p}(x) = +\infty$, it must have some zero; in order to show that this zero is unique, it suffices to show that, as $x$ increases from $0$ to infinity, $\ell_{d,p}$ is negative for some interval and then positive for the remainder.

    To do this, we first consider $m_{d,p}(x)\defeq \psi'(1+x)-p\sum_{j=1}^{d-1} \frac{1}{(x+j)^2}$, so that $\ell_{d,p}(x)=x\ m_{d,p}(x).$ We use the following inequalities on the polygamma functions (which are derivatives of the digamma function). Those can be found in \cite[Lemma 3]{qi2010complete}, say, and hold for any $t\in(0,\infty)$:
    \begin{equation}
    \label{eqn:polygamma-bounds}
        \frac{t+\frac{1}{2}}{t^2} < \psi'(t) < \frac{t+1}{t^2}, \quad
        - \frac{1}{t^2} - \frac{2}{t^3} < \psi''(t) < -\frac{1}{t^2} - \frac{1}{t^3}.
    \end{equation}
    From \eqref{eqn:polygamma-bounds}, we have
    \begin{equation}
    \label{eqn:finite-m-bounds}
        m_{d,p}(x) > \frac{x+\frac{3}{2}}{(x+1)^2} - p \sum_{j=1}^{d-1} \frac{1}{(x+j)^2} \geq \frac{x+\frac{3}{2}}{(x+1)^2} - p\sum_{j=1}^{d-1} \frac{1}{(x+1)^2} = \frac{x-x_\ast}{(x+1)^2}.
    \end{equation}
    Thus $m_{d,p}$ (and therefore $\ell_{d,p}$) is strictly positive for $x \geq x_\ast$. We also note that 
    \[
        m_{d,p}(0) = \psi'(1) - p\sum_{j=1}^{d-1} \frac{1}{j^2} \leq \psi'(1) - 2\sum_{j=1}^{d-1} \frac{1}{j^2} \leq \psi'(1) - 2 = \frac{\pi^2}{6} - 2 < 0.
    \]
    Therefore, in order to show that $h_{d,p}$ has a unique zero, it suffices to show that $m_{d,p}$ is strictly increasing on $[0,x_\ast]$. We do this in three cases according to the values of $d$ and $p$.
    \begin{itemize}
    \item \textbf{Case $1$ ($p \geq 4$):} Using \eqref{eqn:polygamma-bounds}, we find
    \begin{equation}
    \label{eqn:finite-mprime-bounds}
    \begin{split}
        m'_{d,p}(x) &= \psi''(1+x) + 2p\sum_{j=1}^{d-1} \frac{1}{(x+j)^3} > -\frac{1}{(x+1)^2} - \frac{2}{(x+1)^3} + 2p\int_{j=1}^d \frac{1}{(x+j)^3} \diff j \\
        &= \frac{-x^3+(-3-2d-2p+2dp)x^2+(-6d-d^2-3p+2dp+d^2p)x+(d^2p-p-3d^2)}{(1+x)^3(x+d)^2}.
    \end{split}
    \end{equation}
    To show this is positive on $(0,x_\ast)$, it suffices to show that the numerator of the last expression, which we will call $n_{d,p}(x)$, is strictly positive over $(0,x_\ast)$. Since $n_{d,p}(x)$ is a cubic in $x$ with negative leading coefficient, we have that $n'_{d,p}(x)$ is a quadratic in $x$ with negative leading coefficient; thus, if we can show that $n'_{d,p}(0) > 0$, we must have that $n'_{d,p}$ has a unique zero $z$ on $(0,\infty)$, with $n_{d,p}$ increasing on $(0,z)$ and decreasing otherwise. If we can additionally show that $n_{d,p}(0) \geq 0$ and $n_{d,p}(x_\ast) > 0$, then we must have that $n_{d,p} > 0$ on $(0,x_\ast)$. 

    So there remain only three computations:
    \begin{itemize}
        \item Notice that $n_{d,p}(0) = d^2(p-3)-p$, which is nonnegative for all $d \geq 2$ and $p \geq 4$.
        \item Notice also that 
        \[
        n_{d,p}(x_\ast) = \frac{8p^3(d-1)^3 - 4(d-1)^2(3+2d)p^2-2(d-1)(4d^2+2d-13)p - 3(3-2d)^2}{8}.
        \]
        We claim the numerator of this fraction is positive for $d \geq 2$ and $p \geq 4$; indeed, for such $d$, it is a cubic in $p$ with positive leading coefficient, whose value at $p = 4$ is $352d^3-1468d^2+1948d-835 > 0$; whose derivative in $p$, at $p = 4$, is $312d^3-1116d^2+1310d-506 > 0$; and whose second derivative in $p$, at $p = 4$, is $8(d-1)^2(22d-27) > 0$. This means that it is convex and increasing in $p$ for $p \geq 4$, hence positive for such $p$; this means that $n_{d,p}(x_\ast) > 0$ for such $d$ and $p$.
        \item For $d \geq 2$ and $p \geq 4$, we also have
        \begin{align*}
                n'_{d,p}(0) &= p(d+3)(d-1)-d(d+6) \geq 4(d+3)(d-1)-d(d+6) = d(3d+2) - 12 \geq 4 > 0.
        \end{align*}
    \end{itemize}
    \item \textbf{Case $2$ ($p = 3$ and $d \geq 3$):} Here, an analogous argument to the one above shows that $n_{d,3}(x_\ast) > 0$, but for $d \geq 3$ rather than $d \geq 2$. Now we run the argument from above on the interval $[1,x_\ast]$; we have $n_{d,3}(1) = -22+2d(2+d) > 0$ for $d \geq 3$, and $n'_{d,3}(1) = -30+2d(4+d) > 0$ for $d \geq 3$. Thus we conclude that $m'_{d,3}(x)$ is positive on $(1,x_\ast)$, but also $m_{d,3}(1) < 0$, so we know that $m_{d,3}$ has a unique zero over $x > 1$, and it remains only to show that $m_{d,p}$ is negative over $x \in [0,1]$. But for such $x$ we have
    \[
        m_{d,3}(x) \leq \frac{x+2}{(x+1)^2} - 3\sum_{j=1}^{d-1} \frac{1}{(x+j)^2} \leq \frac{x+2}{(x+1)^2} - 3\sum_{j=1}^2 \frac{1}{(x+j)^2} = \frac{x^3-6x-7}{(2+3x+x^2)^2}
    \]
    which is negative for $x \in [0,1]$.
    \item \textbf{Case $3$ ($p = 3$ and $d = 2$):} Here we skip estimating the sum by the integral, and replace \eqref{eqn:finite-mprime-bounds} with the simpler estimate
    \[
        m'_{d,p}(x) = \psi''(1+x) + \frac{6}{(x+1)^3} \geq -\frac{1}{(x+1)^2} - \frac{2}{(x+1)^3} + \frac{6}{(x+1)^3} = \frac{3-x}{(1+x)^3}.
    \]
    This is positive for $x < 3$, which suffices since $x_\ast(2,3) = 3/2$.
    \end{itemize}
    This ends the proof that $h_{d,p}$ has a unique zero.

    To bound this zero, we use \eqref{eqn:finite-m-bounds} to bound $\ell_{d,p}(x)\ge x\frac{(x-x_*)}{(x+1)^2}$. Thus, since $\ell_{d,p}(0) = 0$, for any $y \geq 0$ we have
    \begin{equation*}
        h_{d,p}(y)=\int_0^y\ell_{d,p}(t)\diff t\ge \int_{0}^yt\frac{(t-x_*)}{(t+1)^2}\diff t.
    \end{equation*}
    The right hand side can be integrated explicitly, yielding
    \begin{equation*}
        h_{d,p}(y)\ge (1+x_*+y+1)-\frac{1+x_*}{y+1} -1-(2+x_*)\log(y+1).
    \end{equation*}
In particular, the location of the zero of $h_{d,p}(y)$ is bounded above by the location of the unique zero of $(1+x_*+y+1)-\frac{1+x_*}{y+1} -1-(2+x_*)\log(y+1)$ on $(0,+\infty)$. Changing variables to $\omega:=1+y, \omega\in(1,+\infty)$ we must look for the solution to the equation
\begin{equation*}
    \log(\omega)-\frac{\omega}{2+x_*}+\frac{1+x_*}{2+x_*} \cdot \frac1{\omega}=\frac{x_*}{2+x_*}.
\end{equation*}
Lemma \ref{lem:upper-bound-location-finite-d} with $s=x_*$ provides the needed claim.
\end{proof}

\begin{lemma}\label{lem:upper-bound-location-finite-d}
    Let $s\ge3/2$ be a real number. Consider the equation
    \begin{equation}\label{eq:exact}
    \log(\omega)-\frac{\omega}{2+s}+\frac{1+s}{2+s}\cdot \frac1{\omega}=\frac{s}{2+s}.
    \end{equation}
    This equation has a unique solution $\omega_0$ over the interval $(1,+\infty)$, which additionally lies in the interval $[1+s,w_+]$, where $w_+$ is the unique solution in $[1+s,+\infty)$ to 
    \begin{equation}\label{eq:approx_upper}
        \log(\omega)-\frac1{2+s}\omega=\frac{s}{2+s}.
    \end{equation}
    In particular, this solution can be written explicitly in terms of the real lower branch of the Lambert $\mathcal{W}$ function as
    \begin{equation}\label{eq:upper_bound_explicit_location_finite_d}
    w_+=-(2+s)\mathcal{W}_{-1}\left(-\frac{e^{\frac{s-1}{2+s}}}{2+s}\right).
    \end{equation}
\end{lemma}
\begin{proof}
    Define $\kappa_s:[1,+\infty)\to \R$ by $\omega\mapsto \log(\omega)-\frac{\omega}{2+s}+\frac{1+s}{2+s}\frac1{\omega}-\frac{s}{2+s}.$ Note that $\kappa_s(1)=0$ for all $s$, and for all strictly positive $s$ we have $\lim_{\omega\to \infty}\kappa_s(\omega)=-\infty$. Computing the derivative of $\kappa_s$ we find
    \[
        \kappa_s'(\omega)=\frac{(1+s-\omega)(\omega-1)}{(2+s)\omega^2}.
    \]
    Therefore, $\kappa_s'(\omega)$ is positive on the interval $[1,1+s]$ and negative elsewhere. Therefore and in particular, $\kappa_s$ must have, over $\omega \in (1,+\infty)$, a unique zero at some $\omega_0 > 1+s$. We can now restrict our analysis to the values of $\omega$ in $[1+s,+\infty).$

    Consider $g_s:[1+s,+\infty)\to \R$ defined by $\omega\mapsto \log(\omega)-\frac{\omega}{2+s}+\frac{1}{2+s}-\frac{s}{2+s}$. Notice that $g$ has a unique zero $\omega_\ast$ over $\omega \geq 1+s$, which furthermore satisfies $\omega_\ast > 2+s$. This follows since $g_s(1+s) = \log(1+s)-\frac{2s}{2+s}$ is positive, and since $g'_s(\omega) = \frac{1}{\omega} - \frac{1}{2+s}$ is positive for $\omega < 2+s$ and negative for $\omega > 2+s$. Additionally, it is here important to note that $\frac{1}{2+s}=\frac{1+s}{2+s}\frac1{1+s}>\frac{1+s}{2+s}\frac1{\omega}$ for all $\omega \in [1+s,+\infty).$ Hence, noting that on this latter interval $g_s(\omega)-\kappa_s(\omega)=\frac{1}{2+s}-\frac{1+s}{2+s}\frac1{\omega}\ge0$, we have $g_s(\omega)\ge \kappa_s(\omega).$ Therefore, the unique zero of $g_s(\omega)$ on $[1+s,+\infty)$ upper bounds the unique zero of $\kappa_s(\omega)$ on this interval.

    We end the proof by noticing that 
    $$g_s(\omega)=0\ \Leftrightarrow \ \log(\omega)-\frac{\omega}{2+s}+\frac{1-s}{2+s}=0 \ \Leftrightarrow \ -\frac{\omega}{2+s}e^{-\frac{\omega}{2+s}}=-\frac{e^{\frac{s-1}{2+s}}}{2+s}.$$
    The last writing of this equation, whose solution we denote $w_+$, is solved for $\omega\ge 2+s$ by the lower real branch of the Lambert function as
    \begin{equation}
        w_+=-(2+s)\mathcal{W}_{-1}\left(-\frac{e^{\frac{s-1}{2+s}}}{2+s}\right),
    \end{equation}
    In fact $-\frac{e^{\frac{s-1}{2+s}}}{2+s}>-\frac1{2+s}>-1/e$ for all $s\ge 3/2,$ which implies there are two real roots given by the upper ($\mathcal{W}_0$) and lower ($\mathcal{W}_{-1}$) branches of the Lambert function, satisfying $\mathcal{W}_{-1}(-\frac{e^{\frac{s-1}{2+s}}}{2+s}) < -1 < \mathcal{W}_0(-\frac{e^{\frac{s-1}{2+s}}}{2+s})$. Since the argument above shows that $\frac{-\omega_\ast}{2+s} < -1$, we conclude that the branch must be the lower one, which finishes the proof.
\end{proof}

\begin{remark}
    Interestingly, we can improve on this bound iteratively. In fact, we now know $\omega_0\in[1+s,w_+]$ and using this upper bound, we can construct a function $g^{(1)}_s(\omega):[1+s,w_+]\to \R, \omega \mapsto \log(\omega)-\frac{\omega}{2+s}+\frac{1+s}{2+s}\frac1{w_+}-\frac{s}{2+s}$. Since $\forall \omega\in[1+s,w_+],\  \frac{1+s}{2+s}\frac1{w_+}\le\frac{1+s}{2+s}\frac1{\omega},$ one has $g^{(1)}_s(\omega)\le\kappa_s(\omega)$ on $[1+s,w_+]$. Therefore, the zero $\omega_-^{(1)}$ of $g^{(1)}_s(\omega)$ lower bounds $\omega_0$, and since $\frac{1+s}{2+s}\frac1{w_+}$ is constant, $\omega_-^{(1)}$ can be expressed explicitly in terms of Lambert function again. Using this lower bound we then know $\omega_0\in [\omega_-^{(1)},w_+].$ We can use the lower end of this new interval to define $g^{(2)}_s(\omega):[\omega_-^{(1)},w_+]\to \R, \omega \mapsto \log(\omega)-\frac{\omega}{2+s}+\frac{1+s}{2+s}\frac1{\omega_-^{(1)}}-\frac{s}{2+s},$ and remark now that $g^{(2)}_s(\omega)\ge \kappa_s(\omega)$ which can be used to upper bound $\omega_0$ in the smaller interval $[\omega_-^{(1)},w_+],$ therefore leading to a better upper bound. Finally, pushing this argument further by iterating it allows to prove that $\omega_0$ satisfies an explicit fixed point equation involving the Lambert function and can therefore be found numerically. This is however much more than we need here.
\end{remark}


\section{Comparison with prior bounds}
\label{app:prior-bounds-comparison}

In this section, we compare our new moment method to four others previously used for upper-bounding injective norms: Kac--Rice arguments, Sudakov--Fernique arguments, epsilon-net arguments and PAC-Bayesian arguments.


\subsection{Kac--Rice arguments.}\
The Kac--Rice formula allows one to count critical points of Gaussian processes. If $M$ is a nice set and $f : M \to \R$ is a smooth Gaussian process, then 
\[
    \max\{f(t) : t \in M\} = \inf\{u \in \R : f \text{ has no critical points in }\{t \in M : f(t) \geq u\}\}
\]
and one can therefore use Kac--Rice to locate $\max f$. This is a common argument in spin glasses; see \cite{AdlTay2007, AzaWsc2009} for a textbook treatment, and \cite{auffinger2013random,subag2017complexity} for a classic example of this strategy at work.

If $T$ is a random tensor with i.i.d. Gaussian entries, then its injective norm can be viewed as the maximum of a Gaussian process, and therefore in principle located in this way. Thus some spin-glass results can be retroactively re-interpreted as statements about the injective norm of random tensors, and we gave an upper bound with this strategy in our prior work \cite{dartois2024injective}, for what we would now call the \hyperlink{model:ak}{Model $A_\K$} with Gaussian entries. (After posting \cite{dartois2024injective} to the arXiv, the authors learned that the high-level idea of using the Kac--Rice formula to study the injective norm of a random tensor was proposed, but not carried out, in \cite{Evn2021}. Bates and Sohn recently found matching lower bounds for some of our results \cite{bates2025balanced}, but without Kac--Rice.) Thus one can compare methods.

First, we remark that the Kac--Rice formula deals with Gaussian processes. Thus most of the random-tensor models considered in this paper (such as the bounded-rank \hyperlink{model:bc}{Model $B_\C$}, or the \hyperlink{model:ak}{$A_{\mathbb{K}}$} models with entries that are sub-Gaussian but not Gaussian) are beyond its scope, because it is not clear how to write their injective norms as the maximum of some Gaussian process.

Second, while the Kac--Rice formula is in principle available at finite $d$ and $p$, it requires one to understand the expected absolute value of the determinant of a random matrix of size $p(d-1) \times p(d-1)$. Since we primarily understand \emph{asymptotics} of such determinants as the matrix size tends to infinity (for example from \cite{BenBouMcK2022}), in practice the Kac--Rice formula can only give bounds in the limit either $p \to \infty$ or $d \to \infty$ (or both). 

Compared to these two, the method presented in this paper is more widely applicable: it can treat non-Gaussian models, and give non-asymptotic results. It is also much simpler technically. However, the Kac--Rice formula gives tighter bounds in the limit $d \to \infty$; the bound there is essentially $\sqrt{p}E_0(p)$, where $E_0(p)$ is a constant coming from spin glasses which is known to scale like $\sqrt{\log p}$ as $p \to \infty$. We cannot presently recover $E_0(p)$ with our methods. However, our bounds are only slightly worse than this for each $p$ (see Figure \ref{fig:relative_diff-KR-versus-moments} and Figure \ref{fig:sudakov-fernique} for a comparison), and they have the same scaling $\sqrt{p \log p}$ as $p \to \infty$ (importantly, with constant 1, not just $\OO(\sqrt{p \log p})$.

\begin{figure}
    \centering
    \includegraphics[width=0.6\linewidth]{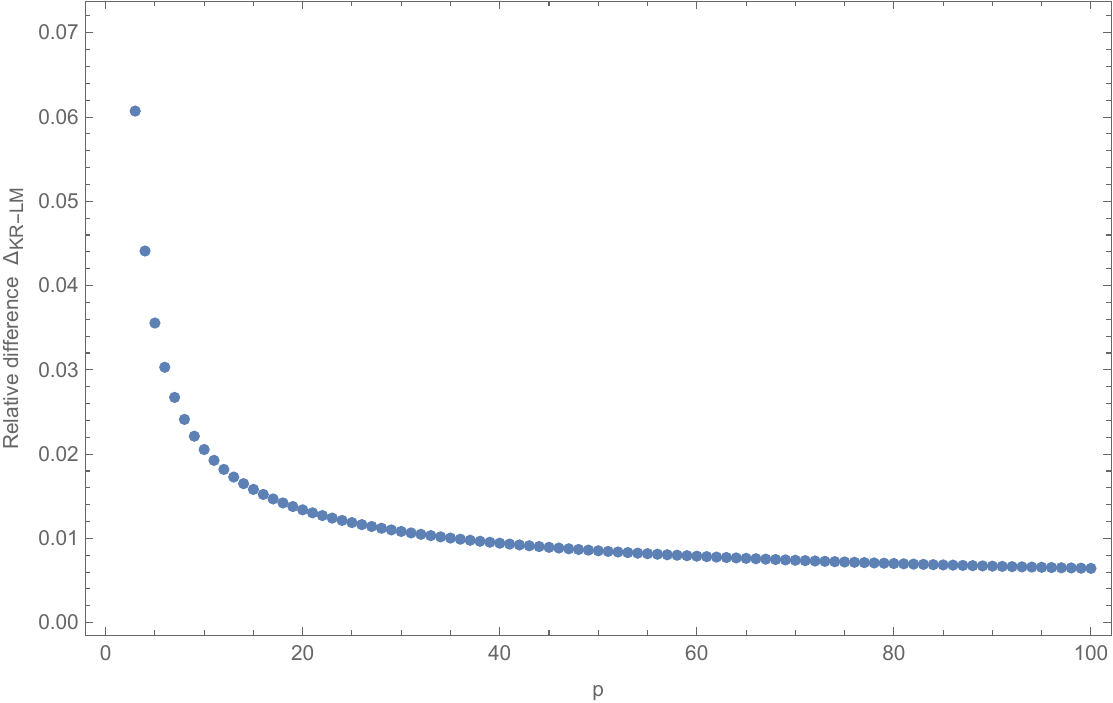}
    \caption{Plot of the relative difference $\Delta_{\text{KR-LM}}:=\frac{\lvert\sqrt{p}E_0(p)-\psi_p(\alpha_0(p)) \rvert}{\sqrt{p}E_0(p)}$ of the bounds obtained by Kac--Rice method versus our large moments approach against $p$, in the case $d_1=\ldots=d_p=d\to\infty$.}
    \label{fig:relative_diff-KR-versus-moments}
\end{figure}


\subsection{Sudakov--Fernique arguments.}\
A classical and well-known argument (dating back at least as far as \cite{DavSza2001}; for a textbook treatment, see, for instance, \cite[Section 7.3]{Ver2018}) uses the Sudakov-Fernique inequality to get a tight upper bound on the operator norm of a random \emph{matrix} with i.i.d. Gaussian entries. (In some places in the literature, the inequality used is called ``Slepian's'' rather than ``Sudakov-Fernique''; elsewhere the term ``Slepian'' is reserved for a Gaussian comparison lemma that requires the variances to be the same.) In this section, we give the natural generalization of this argument to random \emph{tensors}, and show that in many cases of interest it is worse than our bound.

We remark also that Sudakov--Fernique only deals with Gaussian processes, and therefore has the same limitation as Kac--Rice arguments. However, unlike Kac--Rice arguments, Sudakov--Fernique arguments \emph{are} non-asymptotic. 

\begin{lemma}
\label{lem:sudakov-fernique}
\textbf{(Sudakov-Fernique bound)} Let $T$ come from \hyperlink{model:ak}{Model $A_\R$}, and suppose that its entries are not just rigidly sub-Gaussian but actually Gaussian. Then
\begin{equation}
\label{eqn:sudakov-fernique}
    \E[\injnormR{T}] \leq \sum_{j=1}^p \frac{\sqrt{d_j}}{(d_1 \cdot \ldots \cdot d_p)^{1/(2p)}}.
\end{equation}
\end{lemma}

Before proving the bound, we discuss its interpretation. In the matrix case $p = 2$, the Sudakov-Fernique bound \eqref{eqn:sudakov-fernique} becomes $(d_1/d_2)^{1/4} + (d_2/d_1)^{1/4}$, which is tight (and better than our bound). (To apply Sudakov-Fernique, one has to choose a comparison Gaussian process. The classical comparison process for the matrix case has an obvious generalization to tensors, and this is the comparison process we use for Lemma \ref{lem:sudakov-fernique}; we do not rule out the possibility that a different comparison process would give a better bound.)

However, in the tensor case $p \geq 3$, if we take $d_1 = \cdots = d_p = d$ for simplicity, the Sudakov-Fernique bound \eqref{eqn:sudakov-fernique} becomes $p$, whereas the bound from Theorem \ref{thm:asymmetric} scales like $\sqrt{p(\log p)(d-1)/d}$, asymptotically as $p \to \infty$ for $d$ fixed. In particular, the Sudakov-Fernique bound cannot capture the correct scaling in $p$. If instead we fix $p$ and let $d_1 = \cdots = d_p = d \to \infty$, then our bound \eqref{eq:multivariate_bound} becomes
\begin{equation}
\label{eqn:sudakov-fernique-our-bound}
    \lim_{d \to \infty} \E[\injnormR{T}] \leq \inf_{\alpha > 0} \frac{2^{-(p-1)/2}}{\sqrt{e}} \left(\frac{(2+\alpha)^{2+\alpha}}{\alpha^{2+\alpha}}\right)^{1/4} \left( \frac{(2+\alpha)^{2+\alpha}}{\alpha^\alpha}\right)^{(p-1)/4}.
\end{equation}
In Figure \ref{fig:sudakov-fernique}, we evaluate this for $p = 3, 4, \ldots, 8$ and compare it to both the Sudakov-Fernique bound and the bound obtained from the Kac--Rice method in our prior paper \cite{dartois2024injective}. Notice that, while the Sudakov--Fernique bound beats the moment bound for $p = 3$, the moment bound beats the Sudakov--Fernique bound asymptotically in $p$ (Theorem \ref{thm:asymmetric} shows that the moment bound scales like $\sqrt{p \log p}$, like the Kac--Rice bound), and it seems numerically that the moment bound beats the Sudakov--Fernique bound for all $p \geq 4$. 

We remark that this is a special case of \eqref{eq:multivariate_bound} when $\eta_j \equiv 1$; it would be interesting to understand the relationship between our bound and the Sudakov-Fernique bound for more general choices of $\eta_j$.

\begin{figure}
\begin{tabular}{c|c|c|c|c|c|c}
$d \to +\infty$, Gaussian & $p = 3$ & $p = 4$ & $p = 5$ & $p = 6$ & $p = 7$ & $p = 8$ \\ \hline
Kac--Rice bound & $2.87$ & $3.59$ & $4.22$ & $4.80$ & $5.33$ & $5.83$ \\ \hline
Moment bound & $3.04$ & $3.75$ & $4.37$ & $4.94$ & $5.47$ & $5.97$ \\ \hline
Sudakov--Fernique bound & $3$ & $4$ & $5$ & $6$ & $7$ & $8$ \\
\multicolumn{7}{c}{} \\
\multicolumn{7}{c}{} \\
$d = 100$, Gaussian & $p = 3$ & $p = 4$ & $p = 5$ & $p = 6$ & $p = 7$ & $p = 8$ \\ \hline
Kac--Rice bound & --- & --- & --- & --- & --- & --- \\ \hline
Moment bound & $3.03$ & $3.72$ & $4.35$ & $4.91$ & $5.44$ & $5.94$ \\ \hline
Sudakov--Fernique bound & $3$ & $4$ & $5$ & $6$ & $7$ & $8$ \\
\multicolumn{7}{c}{} \\
\multicolumn{7}{c}{} \\
$d = 100$, rigidly sub-Gaussian & $p = 3$ & $p = 4$ & $p = 5$ & $p = 6$ & $p = 7$ & $p = 8$ \\ \hline
Kac--Rice bound & --- & --- & --- & --- & --- & --- \\ \hline
Moment bound & $3.03$ & $3.72$ & $4.35$ & $4.91$ & $5.44$ & $5.94$ \\ \hline
Sudakov--Fernique bound & --- & --- & --- & --- & --- & ---
\end{tabular}
\caption{In this figure, we informally summarize three upper bounds for the quantity $\E[\injnormR{T}]$ in the case that $T$ comes from \protect\hyperlink{model:ak}{Model $A_\R$}, when $p$ is fixed and $d_1 = \cdots = d_p = d$ (in the $d \to +\infty$ table, the bound is for $\limsup_{d \to \infty} \E[\injnormR{T}]$). The ``Kac--Rice bound'' is a numerical evaluation of the quantity referred to as $\alpha(p)$ in \cite[Theorem 1.1]{dartois2024injective}; the ``moment bound'' is a numerical evaluation of the right-hand side of \eqref{eqn:sudakov-fernique-our-bound} in the $d \to \infty$ table, and is a numerical evaluation of the right-hand side of \eqref{eqn:real-case-finite-bound} (divided by $\sqrt{d}$) in the $d = 100$ tables; the ``Sudakov--Fernique bound'' is the right-hand side of \eqref{eqn:sudakov-fernique}.} 
\label{fig:sudakov-fernique}
\end{figure}

\begin{proof}[Proof of Lemma \ref{lem:sudakov-fernique}]
Let $g^{(1)}, \ldots, g^{(p)}$ be independent real Gaussian vectors such that 
\[
    g_j \sim \mc{N}_\R\left(0,\frac{1}{(d_1 \cdot \ldots \cdot d_p)^{1/p}} \mathbbm{1}_{d_j}\right).
\]
Consider the Gaussian processes $X, Y : \bSR^{d_1-1} \times \cdots \times \bSR^{d_p-1} \to \R$ defined by
\begin{align*}
    X(x^{(1)},\ldots,x^{(p)}) &= \ip{T,x^{(1)} \otimes \cdots \otimes x^{(p)}}. \\
    Y(x^{(1)},\ldots,x^{(p)}) &= \ip{g^{(1)},x^{(1)}} + \cdots + \ip{g^{(p)},x^{(p)}}.
\end{align*}
These are centered Gaussian processes, with $\sup X = \injnormR{T}$. We claim that
\begin{align*}
    &\E[(X(x^{(1)},\ldots,x^{(p)}) - X(\widetilde{x}^{(1)},\ldots,\widetilde{x}^{(p)}))^2] = \frac{1}{(d_1 \cdot \ldots \cdot d_p)^{1/p}} \sum_{i_1,\ldots,i_p} (x^{(1)}_{i_1} \cdot \ldots \cdot x^{(p)}_{i_p} - \widetilde{x}^{(1)}_{i_1} \cdot \ldots \cdot \widetilde{x}^{(p)}_{i_p})^2 \\
    &= \frac{1}{(d_1 \cdot \ldots \cdot d_p)^{1/p}} \left(2 - 2 \prod_{j=1}^p \ip{x^{(j)},\widetilde{x}^{(j)}}\right).
\end{align*}
Indeed, the first equality is clear; for the second, one can let $\mathfrak{X}=x^{(1)}\otimes \ldots \otimes x^{(p)}$ and $\tilde{\mathfrak{X}}=\tilde x^{(1)}\otimes \ldots \otimes \tilde x^{(p)}$, and therefore think about $\mathfrak{X},\tilde{\mathfrak{X}}\in \R^{d_1} \otimes \cdots \otimes \R^{d_p}$. One has $\sum_{i_1,\ldots,i_p} (x^{(1)}_{i_1} \cdot \ldots \cdot x^{(p)}_{i_p} - \widetilde{x}^{(1)}_{i_1} \cdot \ldots \cdot \widetilde{x}^{(p)}_{i_p})^2=\lnorm \mathfrak{X}-\tilde{\mathfrak{X}}\rnorm_{\textup{HS}}^2=\lnorm \mathfrak{X}\rnorm_{\textup{HS}}^2+\lnorm \tilde{\mathfrak{X}}\rnorm_{\textup{HS}}^2-2\langle\mathfrak{X} ,\tilde{\mathfrak{X}}\rangle$. Since $\lnorm \mathfrak{X}\rnorm_{\textup{HS}}^2=\lnorm \tilde{\mathfrak{X}}\rnorm_{\textup{HS}}^2=1$ and $\langle\mathfrak{X} ,\tilde{\mathfrak{X}}\rangle=\prod_{j=1}^p \ip{x^{(j)},\widetilde{x}^{(j)}}$, we obtain the claim. Next, we can also compute
\begin{align*}
    &\E[(Y(x^{(1)},\ldots,x^{(p)}) - Y(\widetilde{x}^{(1)},\ldots,\widetilde{x}^{(p)}))^2] = \E\left[ \left( \sum_{j=1}^p \ip{g^{(j)},x^{(j)} - \widetilde{x}^{(j)}} \right)^2 \right] \\
    &= \frac{1}{(d_1 \cdot \ldots \cdot d_p)^{1/p}} \sum_{j=1}^p \|x^{(j)} - \widetilde{x}^{(j)}\|_2^2 = \frac{1}{(d_1 \cdot \ldots \cdot d_p)^{1/p}} \sum_{j=1}^p \left(2 - 2\ip{x^{(j)}, \widetilde{x}^{(j)}}\right).
\end{align*}
By induction, one can easily show that $1 - \prod_{j=1}^p a_j \leq \sum_{j=1}^p (1-a_j)$ whenever $a_j \in [-1,1]$ (in the induction step, one notices that $A = \prod_{j=1}^{p-1} a_j$ is also in $[-1,1]$, so that $1 - \prod_{j=1}^p a_j = 1 - Aa_p \leq (1-A)+(1-a_p)$); this implies 
\[
    \E[(X(x^{(1)},\ldots,x^{(p)}) - X(\widetilde{x}^{(1)},\ldots,\widetilde{x}^{(p)}))^2] \leq \E[(Y(x^{(1)},\ldots,x^{(p)}) - Y(\widetilde{x}^{(1)},\ldots,\widetilde{x}^{(p)}))^2]
\]
for all arguments. Thus, from the Sudakov-Fernique inequality, we conclude that
\begin{equation}
\label{eqn:sudakov-fernique-bound}
\begin{split}
    \E[\injnormR{T}] &= \E\left[\sup_{x^{(1)},\ldots,x^{(p)}} X(x^{(1)},\ldots,x^{(p)})\right] \leq \E\left[\sup_{x^{(1)},\ldots,x^{(p)}} Y(x^{(1)},\ldots,x^{(p)})\right] \\
    &= \sum_{j=1}^p \E\left[\sup_{x^{(j)} \in \mathbb{S}_\R^{d_j-1}} \ip{g^{(j)},x^{(j)}} \right] = \sum_{j=1}^p \E[\|g^{(j)}\|_2] \leq \sum_{j=1}^p \E[\|g^{(j)}\|_2^2]^{1/2}
\end{split}
\end{equation}
and now the right-hand side of \eqref{eqn:sudakov-fernique-bound} equals the right-hand side of \eqref{eqn:sudakov-fernique}.
\end{proof}


\subsection{Prior epsilon-net bounds.}\

Classically, one can bound the maximum of a Lipschitz stochastic process over a continuous parameter space by discretizing the space into a so-called ``epsilon-net.'' In the case of injective norms for random tensors, there are at least two strands of well-known papers carrying out this strategy. One strand, coming from quantum information theory \cite{GroFlaEis2009,friedland2018most}, treats specifically Gaussian cases of our Models \hyperlink{model:ak}{$A_\C$} and \hyperlink{model:sk}{$S_\C$}; another strand, coming from random matrices \cite{bandeira2024geometric,boedihardjo2024injective}, treats generalizations of these models that allow for either variance profiles or some correlations between the entries. As we will show, papers in the second strand treat more general models, but tend to give worse bounds than the first strand when restricted to our models, since they are not optimized for this special case; papers in the first strand tend to capture the correct order but not the correct constant.

We first discuss the quantum-information papers \cite{GroFlaEis2009,friedland2018most}. Epsilon-net arguments can give non-asymptotic bounds; however, in this discussion we will follow the example of these papers in thinking of $d$ fixed ($d = 2$ in \cite{GroFlaEis2009}, $d \geq 2$ in \cite{friedland2018most}) and $p \to \infty$, in the sense of giving bounds that hold for $p$ finite but large enough.

To match \cite{GroFlaEis2009,friedland2018most}, we now switch our perspective in two ways: We focus on complex tensors rather than real ones, and we focus on normalized tensors, which we write as $\ket{\psi_T} = T/\|T\|_{\textup{HS}}$. Notice that $\injnormC{\ket{\psi_T}} = \injnormC{T}/\|T\|_{\textup{HS}}$, and in all our cases of interest the quantity $\|T\|_{\textup{HS}}$ is practically deterministic (as the sum of many independent random variables), so most of the difficulty of the problem resides in the un-normalized injective norm $\injnormC{T}$. See, e.g., \cite[Proof of Theorem 3.4]{dartois2024injective} for details on dealing with the normalization.

The paper \cite{GroFlaEis2009} deals with asymmetric tensors, here corresponding to $\ket{\psi_T}$ where $T$ comes from the asymmetric Gaussian \hyperlink{model:ak}{Model $A_\C$}, specifically in the ``qubit case'' $d = 2$. As discussed in our prior work \cite[Proposition 1.3]{dartois2024injective}, \cite{GroFlaEis2009} essentially shows that, with high probability as $p \to \infty$, we have $\injnormC{\ket{\psi_T}} \leq C\frac{\sqrt{p \log p}}{2^{p/2}}$ for some large $C$ (see also \cite{aubrun2017alice} for a textbook treatment, and a lower bound of the form $\geq c\frac{\sqrt{p\log p}}{2^{p/2}}$ for some small $c$). Our previous Kac--Rice arguments \cite{dartois2024injective} showed that one could take $C = 1+\epsilon$ for any $\epsilon$; since our present method can match the Kac--Rice bound, it can also show that one can take $C = 1+\epsilon$, and in particular gives a tighter bound than \cite{GroFlaEis2009} (albeit in expectation, rather than with high probability).

The paper \cite{friedland2018most} deals with the symmetric analogue of \cite{GroFlaEis2009}, for general $d \geq 2$, which corresponds in our notation to studying $\injnormC{\ket{\psi_B}}$, where $\ket{\psi_B} = B/\|B\|_{\textup{HS}}$ and $B$ comes from symmetric Gaussian \hyperlink{model:sc}{Model $S_\C$}. Their main theorem gives a high-probability upper bound, with a probability tending to one very quickly; by re-purposing their arguments to get a tighter upper bound with probability tending to one less quickly, one can show the following.

\begin{lemma}
\label{lem:friedland-kemp}
\textbf{(Essentially due to \cite{friedland2018most})} For each $\epsilon$ and $d \geq 2$, we have
\[
    \P\left(\injnormC{\ket{\psi_B}} \leq (1+\epsilon) \sqrt{2(d+1)} \cdot \sqrt{\frac{(d-1)!(d-1)(\log p)}{p^{d-1}}} \right) = 1 - \oo_{p \to \infty}(1).
\]
\end{lemma}

(All proofs in this subsection are deferred to the end.) Comparatively, our results can show the following bound.

\begin{corollary}
\label{cor:comparing-friedland-kemp}
For each $\epsilon$ and $d \geq 2$, for $p$ large enough we have
\[
    \E[\injnormC{\ket{\psi_B}}] \leq (1+\epsilon)\sqrt{\frac{(d-1)!(d-1)(\log p)}{p^{d-1}}}.
\]
\end{corollary}

Thus, our methods can beat Lemma \ref{lem:friedland-kemp} by a factor of $\sqrt{2(d+1)}$, which is bigger than one for all $d \geq 2$, and diverges with $d$. However, Lemma \ref{lem:friedland-kemp} is also a high-probability bound, while we have focused in this work on bounding the expectation. 

Now we switch gears to consider results coming from random matrices, returning to the case of unnormalized real tensors. The paper \cite{bandeira2024geometric} treats so-called structured tensors of the form
\begin{equation}
\label{eqn:structured-tensors}
    T = \sum_{k=1}^n g_k T_k,
\end{equation}
where the $g_k$'s are i.i.d. standard Gaussians and the $T_k$'s are real deterministic tensors in $(\R^d)^{\otimes r}$. Any tensor with jointly Gaussian entries can be written in such a way. Bandeira et al. give upper bounds on the so-called \emph{$\ell_p$ injective norms}
\[
    \|T\|_{\mathcal{I}_p} = \sup_{x_1 \in B^{d_1}_p, \ldots, x_r \in B^{d_r}_p} \sum_{i_1=1}^{d_1} \cdots \sum_{i_p=1}^{d_p} T_{i_1,\ldots,i_p} (x_1)_{i_1} \cdot \ldots \cdot (x_p)_{i_p}.
\]
where $B^d_p = \{x \in \R^d : \|x\|_p \leq 1\}$ is now the unit ball in $\ell_p$. In this paper, we only treat the case $p = 2$ (and our $p$ is their $r$); one can recover our Gaussian \hyperlink{model:ak}{Model $A_\R$} by taking the structure tensors $T_k$ to be (rescalings of) the tensor $E_{(i_1,\ldots,i_p)}$ which has a one in the $(i_1,\ldots,i_p)$ position and zeros elsewhere. As a special case of their result, one obtains the following.

\begin{lemma}
\label{lem:bandeira-et-al}
\textbf{(Special case of \cite{bandeira2024geometric})} If $T$ comes from Gaussian \hyperlink{model:sc}{Model $A_\R$} with $d_1 = \cdots = d_p = d$, then there exist constants $C_p$ such that
\[
    \E[\injnormR{T}] \leq C_p \log d.
\]
\end{lemma}

Theorem \ref{thm:asymmetric} gives a version of this bound without the $\log d$ and with an explicit $C_p$.

The paper \cite{boedihardjo2024injective} treats a special case of \eqref{eqn:structured-tensors} where the $T_k$'s have only one nonzero entry, but where the value of this entry can vary across $k$. In other words, the tensors in \cite{boedihardjo2024injective} have independent Gaussian entries with possibly different variances. In the case where all the variances are the same, Boedihardjo's result reduces to the following.

\begin{lemma}
\label{lem:boedihardjo}
\textbf{(Special case of \cite{boedihardjo2024injective})} Let $T$ come from Gaussian \hyperlink{model:ak}{Model $A_\R$} with $d_1 = \cdots = d_p = d$. Then there exists a universal constant $C \geq 1$ such that
\[
    \E[\injnormR{T}] \leq \sqrt{2} p^{3/2} + C \cdot \frac{p^3 (\log d)^2}{\sqrt{d}}.
\]
\end{lemma}

Compared to Lemma \ref{lem:bandeira-et-al}, the scaling in $d$ is better here, and the scaling in $p$ is explicit; however, Theorem \ref{thm:asymmetric} gives a bound of the asymptotic order $\sqrt{(\frac{d-1}{d})p\log p}$, which is better. But we emphasize again that both \cite{bandeira2024geometric} and \cite{boedihardjo2024injective} are interested in more general models than just the independent-entry, fixed-variance \hyperlink{model:ak}{Model $A_\R$}.

\begin{proof}[Proof of Lemma \ref{lem:friedland-kemp}]
Matching notation: Our $\injnormC{T}$ is their $\|T\|_\infty$, and they define $E(T) = -2\log_2(\|T\|_\infty) = -2\log_2(\injnormC{T})$. Our $p$ is their $m$, and our $d$ is their $n$. Translated into our language, their Proposition 3.10 becomes the following: For each $d$, there exists some $K_d \leq 2^{d+1}d^d$ such that
\[
    \P(\injnormC{\ket{\psi_B}} \geq t) \leq K_d \exp\left(2(d-1) \log\left(\frac{p}{t}\right) -\frac{2\binom{d+p-1}{p}-1}{4}t^2\right) = K_d \exp(g_{p,d}(t)).
\]
Our goal is to choose $t = t(p,d)$ as small as possible while still ensuring that the right-hand side tends to zero as $p \to \infty$; in principle this can be computed extremely precisely, but we will only track the leading order in $d$ and $p$. Since we only care about asymptotics, we henceforth ignore $K_d$.

Let 
\[
    t(p,d) = s(p,d) \sqrt{\frac{(d-1)!(d-1)(\log p)}{p^{d-1}}}
\]
for some $s(p,d)$ to be chosen. We have
\[
    \lim_{p \to \infty} \frac{\frac{2\binom{d+p-1}{p}-1}{4} \cdot \frac{(d-1)!(d-1)(\log p)}{p^{d-1}}}{\log p} = \frac{d-1}{2}
\]
and thus for each $\epsilon$ and large enough $p$ we have
\begin{equation}
\label{eqn:friedland-kemp-g-upper-bound}
\begin{split}
    g_{p,d}(t(p,d)) &\leq \left( 2(d-1) - \frac{(1-\epsilon)s(p,d)^2(d-1)}{2} + (d-1)^2 \right) \log p \\
    &\quad - 2(d-1)\log(s(p,d)\sqrt{(d-1)!(d-1)(\log p)}). 
\end{split}
\end{equation}
Suppose first that we choose $s(p,d)$ to depend only on $d$; then this tends to negative infinity as long as we make the coefficient of $\log p$ negative, i.e., as long as $s(p,d) > \sqrt{2(d+1)/(1-\epsilon)}$; by redefining $\epsilon$, we get the desired result, at least within this class of $s(p,d)$. Since we want a tight upper bound, we could only improve it by taking $s$ smaller; but when $s$ decreases, the bound \eqref{eqn:friedland-kemp-g-upper-bound} gets worse, so (up to lower-order corrections which we are ignoring) this is the best we can do. 
\end{proof}

\begin{proof}[Proof of Corollary \ref{cor:comparing-friedland-kemp}]
First, we claim that $\|B\|_{\textup{HS}}$ and $\ket{\psi_B} = \frac{B}{\|B\|_{\textup{HS}}}$ are independent (in the Gaussian case). In the \emph{asymmetric} Gaussian case, this is clear for basically two reasons: A tensor $T$ can be thought of as a long vector whose entries are the tensor entries $T_{i_1,\ldots,i_p}$, or it can be thought of as an element of the finite-dimensional vector space of all asymmetric tensors, written out in the basis $(E_{i_1,\ldots,i_p})_{(i_1,\ldots,i_p)}$ (here $E_{(i_1,\ldots,i_p)}$ is the tensor with a one in position $(i_1,\ldots,i_p)$ and zeros everywhere else), which is orthonormal with respect to the Frobenius inner product. In \emph{either} of these perspectives, the Gaussian \hyperlink{model:ak}{Model $A_\C$} is a complex Gaussian vector with constant-times-identity covariance, and such vectors $X$ always have independent norms $\|X\|_{\textup{HS}}$ and directions $X/\|X\|_{\textup{HS}}$.

In the \emph{symmetric} case considered here, only the latter perspective is the good one. The space $\Sym_p(\C^d)$ has dimension $\binom{d+p-1}{p}$, and one can store such a tensor as a vector of length $\binom{d+p-1}{p}$ consisting of all the entries sufficient to reconstruct the tensor from symmetry. \hyperlink{model:ak}{Model $A_\C$} is a Gaussian vector in this space with a covariance matrix that is \emph{not} constant-times-identity, but also the naive length of a vector in this space no longer corresponds to the Frobenius norm, because it has forgotten the different multiplicities coming from symmetry. Instead, the good strategy is to construct a basis consisting of symmetric tensors like $E_{1,1,1}$ and $(1/\sqrt{3})(E_{1,1,3}+E_{1,3,1}+E_{3,1,1})$; the different normalizations are required to make this an orthonormal basis with respect to the Frobenius inner product. Lengths in this basis correspond to Frobenius norms, and from our choice of normalization of the tensor entries, a Gaussian tensor from \hyperlink{model:sc}{Model $S_\C$} actually has a constant-times-identity covariance in this new basis, and thus has independent norms and directions as claimed.

Thus $\E[\injnormC{B}] = \E[\|B\|_{\textup{HS}}] \E[\injnormC{\ket{\psi_B}}]$, as well as (now recalling our actual normalization) $\E[\|B\|_{\textup{HS}}^2] = \frac{\dim(\Sym_p(\C^d))}{d}$. Concentration of the norm of a Gaussian vector yields 
\[
    \E[\|B\|_{\textup{HS}}] = \sqrt{\frac{\dim(\Sym_p(\C^d))}{d}}(1+\oo_{p \to \infty}(1)),
\]
and since $\dim(\Sym_p(\C^d)) = \binom{d+p-1}{p} = \frac{p^{d-1}}{(d-1)!} (1+\oo_{p \to \infty}(1))$, Theorem \ref{thm:symm_complex_tensors_bound} gives
\[
    \E[\injnormC{\ket{\psi_B}}] = \frac{\E[\injnormC{B}]}{\E[\|B\|_{\textup{HS}}]} \leq (1+\epsilon) \frac{\sqrt{\log p} \sqrt{\frac{d-1}{d}}}{\sqrt{\frac{p^{d-1}}{d!}}} = (1+\epsilon) \sqrt{\frac{(d-1)! (d-1) (\log p)}{p^{d-1}}}.
\]
\end{proof}

\begin{proof}[Proof of Lemma \ref{lem:bandeira-et-al}]
The main results of \cite{bandeira2024geometric} are stated for symmetric tensors. However, their Definition 4.13 gives a way to extend to asymmetric tensors $T$ by applying their results to a so-called ``symmetric embedding''\footnote{The tensor version of the standard Hermitization or symmetrization of a non-symmetric matrix.} of $T$. In our case, the symmetric embedding of $T \in (\R^d)^{\otimes p}$ is a symmetric tensor $\sym(T) \in (\R^{pd})^{\otimes p}$, with entries
\[
    \sym(T)_{i_1,\ldots,i_p} = \sum_{\pi \in \mathfrak{S}_p} \delta_{i_{\pi(1)} \in \llbracket 1, d \rrbracket} \delta_{i_{\pi(2)} \in \llbracket d+1, \ldots, 2d \rrbracket} \cdot \ldots \cdot \delta_{i_{\pi(p)} \in \llbracket (p-1)d+1,pd \rrbracket} T_{i_{\pi(1)},i_{\pi(2)}-d,\ldots,i_{\pi(p)}-(p-1)d}
\]
for $i_1,\ldots,i_p \in \llbracket 1, pd \rrbracket$. Notice that, if we partition $(1,\ldots,pd)$ into the $p$ parts $(1,\ldots,d)$, $(d+1,\ldots,2d)$, and so on, then $\sym(T)_{i_1,\ldots,i_p}$ is nonzero only if each $i_j$ is in a different part; say that such an ordered tuple $(i_1,\ldots,i_p)$ is \emph{good}. In particular, the only nonzero entries $\sym(T)_{i_1,\ldots,i_p}$ with $i_1 \leq \cdots \leq i_p$ are those of the form
\[
    \sym(T)_{i_1,d+i_2,\ldots,(p-1)d+i_p} = T_{i_1,\ldots,i_p}
\]
for $i_1,\ldots,i_p \in \llbracket 1, d \rrbracket$. In particular, the entries of $\sym(T)_{i_1,\ldots,i_p}$ are independent for $i_1 \leq \cdots \le i_p$, and $\E[\sym(T)] = 0$, so Theorem 2.1 of \cite{bandeira2024geometric} applies to this tensor. In the notation there, the ``variance tensor'' $A \in (\R^{pd})^{\otimes p}$ of $\sym(T)$ has entries
\[
    A_{i_1,\ldots,i_p} = \begin{cases} \frac{1}{d} & \text{if } (i_1,\ldots,i_p) \text{ is good}, \\ 0 & \text{otherwise}. \end{cases}
\]
Their Theorem 2.1 is stated in terms of $\ell_1$ injective norms of $A$ and various contractions of $A$; the $\ell_1$ injective norm of any tensor is the largest absolute value of any entry of that tensor (see, e.g., the proof of their Corollary 2.2), so these are easy to compute. For example, $\|A\|_{\mc{I}_1} = 1/d$. With $\mathbf{1} \in \R^{pd}$ the all-ones vector and $q < p$, they write $A\mathbf{1}^{\otimes q}$ for the tensor in $(\R^{pd})^{\otimes (p-q)}$ with entries 
\begin{equation}
\label{eqn:bandeira-a}
    (A\mathbf{1}^{\otimes q})_{i_{q+1},\ldots,i_p} = \sum_{i_1,i_2,\ldots,i_q=1}^{pd} A_{i_1,i_2,\ldots,i_p} = \frac{1}{d} \#\{(i_1,\ldots,i_q) : (i_1,\ldots,i_p) \text{ is good}\}.
\end{equation}
If each entry of $(i_{q+1},\ldots,i_p)$ is in a different part, then the right-hand side of \eqref{eqn:bandeira-a} is $d^{-1} d^q q!$, since there are $q!$ ways to assign the first $q$ indices to parts and then $d$ possible values for each index. Otherwise, the right-hand side of \eqref{eqn:bandeira-a} is zero. When $q = p$, by convention we take $A\mathbf{1}^{\otimes q}$ to be $d^{-1}$ times the scalar counting the total number of good entries, which is $d^pp!$. Thus $\|A\mathbf{1}^{\otimes q}\|_{\mathcal{I}_1} = d^{q-1} q!$ for all $1 \leq q \leq p$, and then Theorem 2.1 of \cite{bandeira2024geometric} says that there exist constants $C_p$ such that
\[
    \E[\injnormR{\sym(T)}] \leq C_p\left( \log d + \max_{2 \leq q \leq p} \left[ d^{\frac{q-1}{2q}} (q!)^{\frac{1}{2q}} d^{-\frac{q-1}{2q}}\right] \right) = C_p \left( \log d + (p!)^{\frac{1}{2p}} \right).
\]
Finally, their Lemma 4.16 shows that $\injnormR{T} = \frac{p^{p/2}}{p!} \injnormR{\sym(T)}$ deterministically, which finishes the proof. (We absorb all the factors of $p$ in the statement, since $C_p$ is anyway not explicit.)
\end{proof}

\begin{proof}[Proof of Lemma \ref{lem:boedihardjo}]
This is an immediate corollary of \cite[Theorem 1.1]{boedihardjo2024injective} where $b_{i_1,\ldots,i_r} = \frac{1}{\sqrt{d}}$ for all $d$. (Notice that Boedihardjo's $r$ is our $p$.)
\end{proof}


\subsection{Prior PAC-Bayesian bounds.}\

Recently, Aden-Ali used the PAC-Bayesian lemma, a kind of variational tool used to control suprema of certain stochastic processes, to study injective norms of random tensors \cite{Ade2025}. Aden-Ali considered $\ell_p$ injective norms of structured tensors of the form \eqref{eqn:structured-tensors}, except that \cite{Ade2025} allows the $g_k$'s to be subgaussian and not necessarily Gaussian. In the special case of our model, Aden-Ali's result reduces to the following.

\begin{lemma}
\label{lem:aden-ali}
\textbf{(Corollary of \cite{Ade2025})} If $T$ comes from \hyperlink{model:ak}{Model $A_\R$}, then
\[
    \E[\injnormR{T}] \leq (d_1 \cdot \ldots \cdot d_p)^{-\frac{1}{2p}} \sqrt{p \left( \sum_{i=1}^p d_i + p \max_{2 \leq \ell \leq p} \left( \sum_{\substack{I \subseteq [p] \\ \abs{I} = p-\ell}} d_{\hat{I}} \right)^{\frac{1}{\ell}}\right)},
\]
where $d_{\hat{I}} = \prod_{j \not\in I} d_j$. In particular, in the case $d_1 = \cdots = d_p = p$, this reduces to
\begin{equation}
\label{eqn:aden-ali-bound}
    \E[\injnormR{T}] \leq \sqrt{p^2 \left(1 + \binom{p}{2}^{1/2} \right)}
\end{equation}
\end{lemma}
The right-hand side of \eqref{eqn:aden-ali-bound} scales like $\sqrt{p^3/2}$ for large $p$; compare to Lemma \ref{lem:boedihardjo}, which also scales like $p^{3/2}$ if we take $d \to \infty$ first. Comparatively, the bound from Theorem \eqref{thm:asymmetric} scales asymptotically like $\sqrt{(\frac{d-1}{d}) p\log p}$, but \cite{Ade2025} is about more general models. 

\begin{proof}[Proof of Lemma \ref{lem:aden-ali}]
If $T$ comes from \hyperlink{model:ak}{Model $A_\R$}, then $\widetilde{T} = (d_1 \cdot \ldots \cdot d_p)^{1/(2p)} T$ can be written as 
\[
    \widetilde{T} = \sum_{(i_1,\ldots,i_p)} \xi_{(i_1,\ldots,i_p) \in [d_1] \times \cdots \times [d_p]} E_{(i_1,\ldots,i_p)},
\]
where $\xi_{(i_1,\ldots,i_p) \in [d_1] \times \cdots \times [d_p]}$ are $\R-$rigidly sub-Gaussian, which implies sub-Gaussianity. Corollary 1.3 of \cite{Ade2025} gives the bound
\begin{equation}
\label{eqn:aden-ali}
    \E[\injnormR{\widetilde{T}}] \leq \sqrt{p \left( \sigma_{(1,2)}^2 + p \max_{2 \leq \ell \leq p} \sigma_{(\ell,2)}^{\frac{2}{\ell}} \sigma_{(0,2)}^{\frac{2\ell-2}{\ell}}\right)},
\end{equation}
where Aden-Ali's notation translates into ours as
\[
    \sigma_{(\ell,2)}^2 = \sigma_{(\ell,p)}^2(T) = \sup_{x_1 \in B_2^{d_1},\ldots,x_p \in B_2^{d_p}} \sum_{\substack{I \subseteq [p] \\ \abs{I} = p-\ell}} \sum_{k=1}^n \|T^{\pi_I}_k(x_{\pi_I(1)},\ldots,x_{\pi_I(p-\ell)},\cdot,\ldots,\cdot)\|_{\textup{HS}}^2.
\]
Here, $A(x_1,\ldots,x_k,\cdot,\ldots,\cdot)$ denotes the tensor obtained by contracting the first $k$ legs of the tensor $A$ with the vectors $x_1,\ldots,x_k$, and $\pi_I$ is an operator on tensors that permutes the indices so that the legs along which one contracts are the \emph{first} ones (see \cite{Ade2025} for details on the notation). In our case, this becomes
\[
    \sigma_{(0,2)}^2 = \sup_{x_1 \in B_2^{d_1},\ldots,x_p \in B_2^{d_p}} \sum_{(i_1,\ldots,i_p)} (x_1)_{i_1}^2 \cdot \ldots \cdot (x_p)_{i_p}^2 = 1
\]
in the case $\ell=0$. For $\ell = 1$ and $p = 3$, say, we have
\[
    \sigma_{(1,2)}^2 = \sup_{x \in B_2^{d_1},y \in B_2^{d_2},z \in B_2^{d_3}} \left( \sum_{(i,j,k)} x_i^2y_j^2 + \sum_{(i,j,k)} x_i^2z_k^2 + \sum_{(i,j,k)} y_j^2z_k^2 \right) = d_1 + d_2 + d_3,
\]
and more generally one can similarly compute
\[
    \sigma_{(\ell,2)}^2 = \sum_{\substack{I \subseteq [p] \\ \abs{I} = p-\ell}} d_{\hat{I}}.
\]
Plugging this into \eqref{eqn:aden-ali} yields the result. When $d_1 = \cdots = d_p = p$, this reduces to
\[
    \E[\injnormR{T}] \leq \frac{1}{\sqrt{d}} \sqrt{p \left(pd + p \max_{2 \leq \ell \leq p} \left( \binom{p}{\ell}\right)^{\frac{1}{\ell}} d \right)},
\]
and the maximum is achieved at $\ell = 2$. 
\end{proof}

\bibliographystyle{alpha}
\bibliography{biblio_simple_bound}

\end{document}